\documentclass[10pt, english]{article}

\usepackage[margin= 1.9 cm]{geometry}

\usepackage{amsthm}
\usepackage{amsmath}
\usepackage{amssymb}
\usepackage{setspace}
\usepackage{mathtools}
\usepackage{verbatim}
\usepackage{csquotes}
\usepackage{graphicx}
\usepackage[hidelinks]{hyperref}
\usepackage{thm-restate}
\usepackage{cleveref}
\usepackage{graphicx}
\usepackage{appendix}
\usepackage[inline]{enumitem}
\usepackage{framed}
\usepackage{subcaption}

\usepackage{bbm}
\usepackage{mathrsfs}

\usepackage{caption}

\usepackage{floatrow}
\usepackage[T1]{fontenc}

\usepackage{tikz}
\usepackage{mathdots}
\usepackage{xcolor}
\usepackage{diagbox}
\usepackage{colortbl}
\usepackage[absolute,overlay]{textpos}

\graphicspath{ {./images/} }

\usetikzlibrary{calc}
\usetikzlibrary{decorations.pathreplacing}
\usetikzlibrary{positioning,patterns}
\usetikzlibrary{arrows,shapes,positioning}
\usetikzlibrary{decorations.markings}

\tikzstyle{edge}=[very thick]
\definecolor{bostonuniversityred}{rgb}{0.8, 0.0, 0.0}
\definecolor{arsenic}{rgb}{0.23, 0.27, 0.29}
\tikzstyle{diredge}=[postaction={decorate,decoration={markings,
		mark=at position .95 with {\arrow[scale = 1]{stealth};}}}]

\newcommand{\defPt}[3]{
	\def \pt {(#1, #2)}
	\coordinate [at = \pt, name = #3];
}

\tikzset{
   conn/.pic={
     \defPt{0.2}{-0.5}{q0}
     \defPt{-1}{-1.5}{q5}
    \defPt{1}{1.2}{q1}
    \defPt{1}{2.7}{q6}
    \defPt{1.25}{-1.2}{q2}
    \defPt{2.5}{0.6}{q3}
    \defPt{2.5}{-0.6}{q4}
  
        \draw[line width=1 pt] (q0) -- (q1) -- (q3) -- (q4);
        \draw[line width=1 pt] (q2) -- (q3);
        \draw[line width=1 pt] (q0) -- (q5);
        \draw[line width=1 pt] (q1) -- (q6);
  }
}

\newcommand{\fitellipsis}[3] % first and second node names without parentheses
{\draw []let \p1=(#1), \p2=(#2), \n1={atan2(\y2-\y1,\x2-\x1)}, \n2={veclen(\y2-\y1,\x2-\x1)}
    in ($ (\p1)!0.5!(\p2) $) ellipse [ x radius=\n2/2+0.3cm+#3cm, y radius=#3cm, rotate=\n1];
}

\newcommand{\fitellipsiss}[3] % first and second node names without parentheses
{\draw [fill=white]let \p1=(#1), \p2=(#2), \n1={atan2(\y2-\y1,\x2-\x1)}, \n2={veclen(\y2-\y1,\x2-\x1)}
    in ($ (\p1)!0.5!(\p2) $) ellipse [ x radius=\n2/2+#3cm, y radius=#3cm, rotate=\n1];
}

\newcommand{\fitellipsisss}[3] % first and second node names without parentheses
{\draw []let \p1=(#1), \p2=(#2), \n1={atan2(\y2-\y1,\x2-\x1)}, \n2={veclen(\y2-\y1,\x2-\x1)}
    in ($ (\p1)!0.5!(\p2) $) ellipse [ x radius=\n2/2+#3cm, y radius=#3cm, rotate=\n1];
}

\theoremstyle{plain}

\newtheorem*{thm*}{Theorem}
\newtheorem{thm}{Theorem}
\Crefname{thm}{Theorem}{Theorems}
\numberwithin{thm}{section}

\newtheorem*{lem*}{Lemma}
\newtheorem{lem}[thm]{Lemma}
\Crefname{lem}{Lemma}{Lemmas}

\newtheorem*{claim*}{Claim}

\Crefname{claim}{Claim}{Claims}
\Crefname{claim}{Claim}{Claims}

\Crefname{prop}{Proposition}{Propositions}

\newtheorem{cor}[thm]{Corollary}
\Crefname{cor}{Corollary}{Corollaries}

\Crefname{conj}{Conjecture}{Conjectures}

\Crefname{qn}{Question}{Questions}

\Crefname{obs}{Observation}{Observations}

\newtheorem{ex}[thm]{Example}
\Crefname{ex}{Example}{Examples}

\theoremstyle{definition}

\Crefname{prob}{Problem}{Problems}

\Crefname{defn}{Definition}{Definitions}

\newtheorem*{defn*}{Definition}

\theoremstyle{remark}

\renewenvironment{proof}[1][]{\begin{trivlist}
\item[\hspace{\labelsep}{\bf\noindent Proof#1.\/}] }{\qed\end{trivlist}}

\newenvironment{cla_proof}[1][]{\begin{trivlist}
\item[\hspace{\labelsep}{\noindent \emph{Proof#1.}\/}] }{\qed\end{trivlist}}

\newcommand{\ceil}[1]{
    \left\lceil #1 \right\rceil
}
\newcommand{\floor}[1]{
    \left\lfloor #1 \right\rfloor
}

\newcommand{\eps}{\varepsilon}

\expandafter\def\expandafter\normalsize\expandafter{%
    \normalsize
    \setlength\abovedisplayskip{8pt}
    \setlength\belowdisplayskip{8pt}
    \setlength\abovedisplayshortskip{4pt}
    \setlength\belowdisplayshortskip{4pt}
}

\usepackage[square,sort,comma,numbers]{natbib}
 \setlist[itemize]{leftmargin=*}

\DeclareFontFamily{OT1}{pzc}{}
\DeclareFontShape{OT1}{pzc}{m}{it}{<-> s * [1.10] pzcmi7t}{}
\DeclareMathAlphabet{\mathpzc}{OT1}{pzc}{m}{it}

\newcommand{\R}{\ensuremath{\mathbb{R}}}
\newcommand{\Ce}{\ensuremath{\mathscr{C}}}
\newcommand{\one}{\ensuremath{\mathbbm{1}}}
\newcommand{\bangle}[1]{\left\langle #1 \right\rangle}
\newcommand{\inprod}[2]{\bangle{#1, #2}}
\DeclareMathOperator{\rk}{rank}
\DeclareMathOperator{\tr}{tr}
\DeclareMathOperator{\MC}{MC}

\title{\vspace{-0.8cm} 
Equiangular lines via improved eigenvalue multiplicity}

\author{Igor Balla\thanks{Department of Mathematics, Masaryk University -- Brno, Czech Republic. Email: \href{mailto:iballa1990@gmail.com} {\nolinkurl{iballa1990@gmail.com}}. Research supported by the MUNI Award in Science and Humanities (MUNI/I/1677/2018).}
\and
Matija Buci\'c\thanks{Department of Mathematics, Princeton University -- Princeton, USA. Email: \href{mailto:mb5225@princeton.edu} {\nolinkurl{mb5225@princeton.edu}}. Research supported in part by an NSF Grant DMS--2349013.}
}
 \date{}

\begin{document}

\maketitle

\vspace{-0.5cm}
\begin{abstract}
A family of lines passing through the origin in an inner product space is said to be \emph{equiangular} if every pair of lines defines the same angle. In 1973, Lemmens and Seidel raised what has since become a central question in the study of equiangular lines in Euclidean spaces. 
They asked for the maximum number of equiangular lines in $\mathbb{R}^r$ with a common angle of $\arccos{\frac{1}{2k-1}}$ for any 
integer $k \geq 2$. We show that the answer equals $r-1+\floor{\frac{r-1}{k-1}},$ provided that $r$ is at least exponential in a polynomial in $k$. This improves upon a recent breakthrough of Jiang, Tidor, Yao, Zhang, and Zhao~[Ann.\ of Math.\ (2) 194 (2021), no. 3, 729--743], who showed that this holds for $r$ at least doubly exponential in a polynomial in $k$. We also show that for any common angle $\arccos{\alpha}$, the answer equals $r+o(r)$ already when $r$ is superpolynomial in $1/\alpha \to \infty$. 

The key new ingredient underlying our results is an improved upper bound on the multiplicity of the second-largest eigenvalue of a graph. In one of the regimes, this improves and significantly extends a result of McKenzie, Rasmussen, and Srivastava [STOC 2021, pp. 396--407]. 
\end{abstract} 

\section{Introduction} 

\subsection{Equiangular lines}
A family of lines passing through the origin in an inner product space is said to be \emph{equiangular} if every pair defines the same angle. The study of large families of equiangular lines dates back to 1947 and it has since found numerous interesting connections and applications in a wide variety of different areas. These include elliptic geometry \cite{HS47, B53}, polytope theory \cite{C73}, frame theory \cite{HS03, CKP13, HC17}, Banach spaces \cite{B19, KLL83, FS17, DL23}, and quantum information theory \cite{Z11, RBSC04, FHS17, FS19, CFS02, FS03} with connections to algebraic number theory and Hilbert's twelfth problem \cite{B17, AFMY17, AFMY20, ABGHM22}. 

Determining the maximum number of equiangular lines in $\R^r$ is widely considered as one of the founding problems of algebraic graph theory \cite{GR01} and features in most introductory courses on the subject. The well-known absolute bound (see e.g.\ \cite{LS73}) states that the answer is at most $\binom{r+1}{2}$, which is tight up to a multiplicative constant and it is a difficult open problem to determine whether there are infinitely many $r$ for which it holds with equality.

The more refined version of this question, where we additionally specify the common angle between the lines, was first posed by Lemmens and Seidel \cite{LS73} in 1973. 
Given $\alpha \in [0,1)$, we will refer to a family of lines passing through the origin as \emph{$\alpha$-equiangular} if the acute angle between any two lines equals $\arccos{\alpha}$. Lemmens and Seidel also gave a good partial answer to their question when $r \leq 1/\alpha^2 - 2$, but until recently, our understanding of the complementary regime $r \geq 1/\alpha^2 - 2$ has been much more limited. 
A series of recent breakthroughs have led to an almost complete solution to this problem when $\alpha$ is fixed and $r$ is sufficiently large. The first result in this series was due to Bukh \cite{B16}, who showed that the answer is $\Theta(r)$ when $r$ is at least exponential in $1/\alpha^2$. Afterward, the first author together with Dr\"{a}xler, Keevash, and Sudakov \cite{igor-felix-paper} determined that the maximum over all $\alpha$ is $2r-2$ in this exponential regime, with equality if and only if $\alpha = 1/3$. In particular, they developed a framework which opened the way for subsequent works to answer this question for specific $\alpha$. Indeed, further improvements for certain $\alpha$ were then obtained by Jiang and Polyanskii \cite{JP20} and in a remarkable work, Jiang, Tidor, Yao, Zhang, and Zhao \cite{yufei-paper} determined the answer for any $\alpha$ and $r$ sufficiently large, provided that there exists a finite graph with largest eigenvalue equal to $\frac{1-\alpha}{2\alpha}$.
If $k$ denotes the smallest order of such a graph (called the spectral radius order of $\frac{1-\alpha}{2\alpha}$), their approach requires $r$ to be at least doubly exponential in $k/\alpha$ and naturally, they asked how large $r$ needs to be for their result to hold. See also a recent survey by Zhao on the topic \cite{notices}.

In the subexponential regime, besides the aforementioned bound of Lemmens and Seidel from 1973, the only upper bounds previous known for general $\alpha$ were obtained by Yu \cite{Y17} and Glazyrin and Yu \cite{GY18} using semidefinite programming and properties of Gegenbauer polynomials. In particular, Yu's bound gives the right order of magnitude for infinitely many $\alpha$ and it is tight whenever the absolute bound is met in dimension $1/\alpha^2 - 2$, but it only applies when $r = \Theta(1/\alpha^2)$. In \cite{igor-paper}, the second author extended the bound of Yu to all $r \le r_{\alpha} \approx \frac{1}{4\alpha^4}$ and showed that the linear bound of \cite{igor-felix-paper} holds for all $r\ge r_{\alpha}$, thereby extending it from the exponential to the polynomial regime.

Our first result determines the answer up to lower order terms provided $r$ is superpolynomial in $1/\alpha \to \infty$.
\begin{thm}\label{thm:superpolynomial}
    If $\alpha \rightarrow 0$ and $r\ge {1/\alpha^{\omega(1)}}$, then the maximum number of $\alpha$-equiangular lines in $\mathbb{R}^r$ is
$$r+o(r).$$
\end{thm}
We note that the assumption that $1/\alpha$ is growing is necessary due to the standard construction in the classical case when $1/\alpha$ is an odd integer, which we describe at the end of \Cref{sec:equiangular}. %In fact, it is necessary whenever there exists a finite graph with largest eigenvalue equal to $\frac{1-\alpha}{2\alpha}$. 
Furthermore, we note that the results of \cite{yufei-paper} establish this theorem when $r$ is at least doubly exponential in $1/\alpha^2$, so the main new contribution of the above theorem is an improved requirement on $r$ from double exponential to only super polynomial in $1/\alpha$. When $r$ is at most exponential in $1/\alpha$, our bound is of the form $r+r/(\log_{1/\alpha} r)^{1-o(1)}$ and in fact, it already beats the previously best known (linear) bound when $r$ is a large polynomial in $1/\alpha$. When $r$ is at least exponential in $1/\alpha$, our bound is of the form $r+r\alpha^{1-o(1)}$. 

In the exponential regime, i.e.\ when $r$ is at least exponential in $1/\alpha^{O(1)}$, we can prove an even more precise result which turns out to be tight. For context, note that here \cite{yufei-paper} gives an upper bound of $$r-1+ (r-1)\max\left\{\frac{C/\alpha}{\log \log r},\frac{1}{k-1}\right\},$$
for some absolute constant $C>0$, where $k$ is the spectral radius order of $\frac{1-\alpha}{2\alpha}$, defined above. Note that unless $r$ is doubly exponential in $1/\alpha$, the first term in this bound is worse than the linear bounds from \cite{igor-felix-paper,JP20}. In \cite{igor-paper}, by replacing the use of Ramsey's theorem in the argument from \cite{yufei-paper} with new tools from spectral graph theory, the two bounds are unified by showing an improved upper bound of 
\begin{equation}\label{eq:igor-bound}
r-1+ (r-1)\max\left\{\frac{C\log (1/\alpha)}{\log \log r},\frac{1}{k-1}\right\}.
\end{equation}
Our main result is the following improvement to \eqref{eq:igor-bound}, giving essentially the correct dependency on $\alpha$ in this regime. 
 
\begin{thm}\label{thm:main}
Given $0< \alpha <1 $, if $r\ge 2^{1/\alpha^{O(1)}}$, then the number of $\alpha$-equiangular lines in $\mathbb{R}^r$ is at most
$$r-1+\floor{(r-1)\cdot \frac{2\alpha}{1-\alpha}}.$$
\end{thm}

The standard construction in the classical case, when the angle is one over an odd integer, shows that our theorem is tight. In particular, we get the following corollary. 

\begin{cor}\label{cor:main}
Given an integer $k \geq 2$, if $r\ge 2^{k^{O(1)}}$, then the maximum number of $\frac{1}{2k-1}$-equiangular lines in $\R^r$ equals
    $$r-1 + \floor{\frac{r-1}{k-1}}.$$
\end{cor}

This improves upon the previously best-known doubly exponential constraint on $r$ from \cite{yufei-paper,igor-paper}. We note that the case where $1/\alpha$ is an odd integer has historically been the primary focus in this area, originating with the seminal work of Lemmens and Seidel in 1973 \cite{LS73}, where they resolved the problem for $k=2$ and sufficiently large $r$. Indeed, they were particularly interested in this case due to a result of Neumann (see \cite{LS73}), who showed that an upper bound of $2r$ holds whenever $1/\alpha$ is not an odd integer. %In view of this result, the main novel contributions of \cite{igor-felix-paper, B17} were also in the case of $1/\alpha$ being an odd integer. 
Finally, we note that for the case of $k=3$, Lemmens and Seidel conjectured how large $r$ needs to be for the answer given in \Cref{cor:main} to kick in. This conjecture has been settled recently by Cao, Koolen, Lin, and Yu \cite{CKLY22}, building on the work of Neumaier \cite{N89}.

\subsection{Second eigenvalue multiplicity} \label{sec:second-eigenvalue}
A crucial ingredient in the quest to solve the equiangular lines problem, which was brought to the table in \cite{yufei-paper}, is a relation between it and the multiplicity of the second largest eigenvalue of a certain auxiliary graph. Our proofs build on this concept, as well as a slightly weaker version of the degree bound obtained in \cite{igor-paper}, while making use of the general framework introduced in \cite{igor-felix-paper}. We note that our main novel contribution is an improved bound for the second eigenvalue multiplicity. However, the nature of the improvement requires a slight change to the general approach. 

In addition to the relation to the equiangular lines problem, the question of bounding the multiplicity of the second eigenvalue of a graph has a number of other connections and is a very interesting problem in spectral graph theory in its own right. For example, the case of Cayley graphs is already quite interesting due to it having connections to deep results in Riemannian geometry. For instance, following the approach of Colding and Minicozzi \cite{colding} on harmonic functions on manifolds and Kleiner’s proof \cite{kleiner} of Gromov’s theorem on groups of polynomial growth \cite{gromov}, Lee and Makarychev \cite{lee2008eigenvalue} showed that in groups with bounded doubling constant, the second eigenvalue multiplicity is bounded. 
Owing to the fact that many algorithmic problems become much easier on graphs with few large eigenvalues, there are also interesting connections to computer science. For instance, McKenzie, Rasmussen, and Srivastava \cite{tcs-paper} do an excellent job of motivating the problem from this perspective and give a detailed history. In addition, they mention further connections to higher-order Cheeger inequalities, typical support size of random walks, and the properties of the Perron eigenvector. Finally, the second eigenvalue multiplicity question is also relevant in the study of Schr\"odinger operators on two-dimensional Riemannian manifolds. Indeed, the connection between the spectral theory of graphs and Schr\"odinger operators on surfaces is a point of view emphasized by Colin de Verdi\`ere and is the underlying motivation for the definition of his graph parameter. In particular, Letrouit and Machado \cite{LM24} were recently able to make progress on a conjecture of de Verdi\`ere \cite{C87} by extending the ideas of \cite{yufei-paper} to Laplacians of negatively curved surfaces.

Our main result here is the following bound on the second eigenvalue multiplicity of an arbitrary graph. 
\begin{thm}\label{thm:multiplicity-main}
    Let $G$ be an $n$-vertex graph with second eigenvalue $\lambda_2$ and maximum degree $\Delta\ge 2.$  Then, 
    $$m_G\left(\lambda_2\right)\le \max\left\{\frac{n}{\lambda_2^{1-o(1)}},\frac{n}{(\log_{\Delta} n)^{1-o(1)}}\right\}.$$
    Moreover, if $\log_{\Delta} n \ge \lambda_2^{O(1)}$, we have $$m_G\left(\lambda_2\right)\le \frac{n}{\lambda_2+1}+n^{o(1)}.$$
\end{thm}
Note that the latter result is essentially tight for the graph consisting of a disjoint union of cliques on $\lambda_2+1$ vertices (and a slight addtive error term is necessary, as we can replace one component with a Paley graph of order $(2\lambda_2+1)^2$). We also note that the first term in the maximum is only relevant when $\log n = \lambda_2^{\Theta(1)}$, as otherwise the second term or the moreover part are stronger. 
This is due to the fact that the above result combines two different approaches,
a more structured one when $\log n \ge \lambda_2^{O(1)}$ (henceforth referred to as the superexponential regime) and a more probabilistic one when $\log n \le \lambda_2$ (henceforth referred to as the subexponential regime).

In the subexponential regime, our bound improves upon and significantly extends a result of McKenzie, Rasmussen, and Srivastava \cite{tcs-paper}, who proved an upper bound of the form $\frac{n (\log d)^{1/4}}{(\log_{d} n)^{1/4-o(1)}}$ for a $d$-regular graph with $d = \Omega(\log^{1/4} n)$. In particular, removing the regularity assumption is key for our applications to the equiangular lines problem. Besides this, their bound becomes trivial already when $d \ge 2^{\Omega(\sqrt{\log n})}$, whereas our result remains non-trivial all the way up to $d=n^{\Omega(1)}.$ In fact, our bound here is the first non-trivial upper bound on the second eigenvalue multiplicity for arbitrary $n$-vertex graphs with maximum degree larger than $\log n$.

While the bound of \Cref{thm:multiplicity-main} in the subexponential regime is behind our proof of \Cref{thm:superpolynomial}, the bound it gives in the superexponential regime is just slightly insufficient to conclude \Cref{thm:main}. The additional ingredient that comes to our rescue in the equiangular lines setting, is that we can reduce to working with connected graphs. Note that Jiang, Yao, Tidor, Zhang, and Zhao \cite{yufei-paper} also made use of this assumption, showing that in an $n$-vertex connected graph of maximum degree $\Delta$, the second eigenvalue multiplicity is at most $O(n/\log_{\Delta} \log n).$ They also observed that in expander graphs, one can improve upon this bound substantially. The key behind our new bound on equiangular lines in the superexponential regime, \Cref{thm:main}, is an extension of this result which is much more flexible and in particular, gives an interesting bound even when using extremely weak expansion (see e.g.\ \Cref{cor:multiplicity-bound} below). Note that the parameter $B$ in the following can be interpreted as a measure of some kind of expansion of $G$.

\begin{thm}\label{thm:multiplicity-bound}
    Let $G$ be a connected $n$-vertex graph with second eigenvalue $\lambda_2\ge 1$ and maximum degree $\Delta$. Suppose that for some $B > 1$ and integer $r \ge 1$, any vertex $v$ of $G$ with degree larger than $\lambda_2-1$ is within distance $r$ from at least $B$ vertices of $G$. Then, provided $2^9r\lambda_2^{4r+2} \log B \le \log_{\Delta} n $, we have
    $$m_G\left(\lambda_2\right)\le \frac{n}{B-1}.$$ 
\end{thm}

We note that for our application to the equiangular lines problem, we only require the above result for a very small value of $r$, namely $r=4$. For comparison to our lower bound requirement on $n$, note that in order for the bound from \cite{yufei-paper} to be non-trivial, one needs $\Delta^{O(1)}\le \log n$. 
The following corollary shows our improved bound in this regime. 
It exploits the very weak expansion we obtain from just the connectivity assumption, capturing the aforementioned result from \cite{yufei-paper} and giving an improvement when $\lambda_2 \gg \log_{\Delta} \log n$ or $\delta \gg 1.$ 

\begin{cor}\label{cor:multiplicity-bound}
    Let $G$ be a connected $n$-vertex graph with second eigenvalue $\lambda_2,$ $\Delta=\Delta(G)\ge 2$, and $\delta=\delta(G)$.  Then, provided $\log n \ge \Delta^{O(1)}$, we have
    $$m_G\left(\lambda_2\right)\le \frac{n}{\lambda_2+\Omega(\delta \log_{\Delta} \log n)}.$$ 
\end{cor}

We note that rather than just bounding the second eigenvalue multiplicity, our arguments can be modified to give a bound on the number of eigenvalues which are close to the second largest eigenvalue. In this approximate version of the problem, one cannot, in general, improve \Cref{cor:multiplicity-bound} by more than a constant factor since Haiman, Schildkraut, Zhang and Zhao \cite{schildkraut-paper} constructed connected graphs with bounded degree (so also bounded $\delta$ and $\lambda_2$) which have at least $\Omega\left(\frac{n}{\log \log n}\right)$ eigenvalues close to $\lambda_2$. On the other hand, for counting the actual second eigenvalue multiplicity of a connected graph on $n$ vertices with bounded degree, the best-known construction, due to \cite{schildkraut-paper}, only has a multiplicity of $n^{1/2-o(1)}$. Similarly, one can prove an approximate version of the bound of \Cref{thm:multiplicity-main} in the subexponential regime and this is tight up to a factor of roughly$\sqrt{\log_\Delta{n}}$, since McKenzie, Rasmussen, and Srivastava \cite{tcs-paper} constructed graphs with the number of eigenvalues close to $\lambda_2$ being at least $\frac{n}{\log_{\Delta}^{3/2-o(1)}n}.$

Finally, we want to mention that while our methods for bounding the second eigenvalue multiplicity certainly build on the framework and ideas introduced by a whole series of remarkable works on the topic, we do bring to the table several new ideas including numerous subtle but crucial changes to the previous arguments. Firstly, the arguments in the subexponential and superexponential regimes are quite different with the former relying more on probabilistic arguments, while the latter on more structured arguments. Similar to previous works \cite{yufei-paper,tcs-paper}, our arguments rely on the trace method combined with control on the growth of the largest eigenvalue of a graph upon the addition of a few vertices and edges. We note that this question is very classical in spectral graph theory in its own right and in fact, one can use a result of Nikiforov \cite{nikiforov-eigenvalue-growth} from 2007 to recover the results in this direction from both \cite{yufei-paper,tcs-paper}. Since Nikiforov's result is known to be tight, this precisely highlights the difficulty in breaking the barriers hit by previous works. Our first key lemma (\Cref{lem:partial-net}) gives a conditional improvement over the aforementioned result of Nikiforov which allows us to prove much better bounds under mild assumptions on the graph (in fact exponentially better in the case relevant for us, after carefully cleaning the graph to ensure we can apply our lemma). The second key lemma (\Cref{lem:small-eval}), despite its simplicity, turns out to be very useful as it simplifies arguments significantly and it is, in particular, crucial in the subexponential regime. It gives a very strong upper bound on the multiplicity of the second eigenvalue, provided that one can find a small subgraph with a large enough first eigenvalue. We believe that both of these lemmas hold great potential for further applications to eigenvalue multiplicity problems as well as to numerous closely related questions.

\subsection{Notation and preliminaries.}
Given a graph $G$, we denote by $V(G)$ and $E(G)$, its vertex and edge set, respectively. Given a vertex $v \in V(G)$, we denote its neighborhood by $N_G(v)$ and denote its degree by $d_G(v) = |N_G(v)|$. We let $\Delta(G)$ and $\delta(G)$ denote the maximum degree and minimum degree of $G$, respectively. Given a subset $S \subseteq V(G)$, we let $G[S]$ denote the induced subgraph of $G$ with vertex set $S$ and let $G \setminus S=G[V(G) \setminus S]$. Moreover, for a subgraph $H \subseteq G$, we let $G \setminus H=G[V(G) \setminus V(H)].$ We let $B_G^{(r)}(v)$ denote the ball of radius $r$ around a vertex $v$ in a graph $G$, namely the subgraph of $G$ induced on the set of all vertices of $G$ which can be reached by paths of length at most $r$ starting at $v$.
We denote by $A_G$, the adjacency matrix of $G$, and by $\lambda_1(G) \ge \lambda_2(G) \ge \ldots \ge \lambda_n(G)$ the eigenvalues of $G$ (so of $A_G$). We let $m_G(\lambda)$ denote the multiplicity of $\lambda$ as an eigenvalue of $G$ and let $m_G([a,b])$ denote the number of eigenvalues of $G$ which belong to the interval $[a,b]$ (counted with multiplicities).
We also let $\one$ denote the all ones vector and for $x \in S$, let $\textbf{e}_x \in \R^S$ denote the standard basis vector in the direction of $x$. We let $I_n, J_n$ denote the $n \times n$ identity and all ones matrices, respectively.
All of our logarithms are in base two unless otherwise specified. 
We will use standard results from linear algebra, such as the rank-nullity theorem, which appear in most undergraduate textbooks on the subject. We also use two standard tools from spectral graph theory. The first is the Cauchy interlacing theorem (sometimes also called Poincaré separation theorem) which, in its graph-theoretic formulation, states that if $H$ is an $m$-vertex induced subgraph of an $n$ vertex graph $G$, then 
$\lambda_{i}(G) \ge \lambda_i(H) \ge \lambda_{n-m+i}(G).$ We also use several parts of the Perron-Frobenius theorem. Namely, that for any graph $G$, one can always choose an eigenvector corresponding to $\lambda_1(G)$ with nonnegative coordinates, and if $G$ is connected, then $\lambda_1(G)>\lambda_2(G)$ and the eigenvector corresponding to $\lambda_1(G)$ can be taken to have all coordinates positive (referred to as the Perron eigenvector).

\section{New bounds on the second eigenvalue multiplicity}

The following lemma is our key new technical tool behind the proof of \Cref{thm:main}. It belongs to a whole class of results in spectral graph theory which attempt to understand the growth rate of the largest eigenvalue in a graph when it is extended by several vertices or edges. Perhaps the closest result in literature is a result of Nikiforov \cite{nikiforov-eigenvalue-growth} from 2007 in which he proves a lower bound on the increase in the largest eigenvalue when extending a graph into a connected one. His result is essentially tight and can be used to give an alternative proof of the results from \cite{yufei-paper}. Our lemma leverages additional information, which we establish as part of the argument, in order to achieve significantly better results than would be possible in the general case. It is also an extension of \cite[Lemma 4.3]{yufei-paper}, which corresponds to the case of $L=\emptyset$ below. 

At a high level, the lemma says that given a graph $K$, if we can split the vertices of a subgraph $F$ into two sets $L$ and $C$ such vertices in $L$ have (relatively) small degree (or even just $\lambda_1(F[L])$ is small) and vertices in $C$ which are all close to $K \setminus F$, then $\lambda_1(K)$ is noticeably larger than $\lambda_1(F)$.

\begin{lem} \label{lem:partial-net}
Let $0<\eps \le 1/4$. Given a non-empty graph $K$ and its subgraph $F$ with at least one vertex such that there is a partition $V(F)=C \sqcup L$ with every vertex in $C$ being within distance $\ell$ from $V(K)\setminus V(F)$ and  
$\lambda_1(F[L])\le \lambda_1(F) \left(1 -{4\eps}\right)$ (where if $L=\emptyset$, we may take $\eps=1$ in the following inequality). Then, $$\lambda_1(F)^{2\ell}\le \lambda_1(K)^{2\ell}-\eps^2.$$  
\end{lem}
\begin{proof}
    If $\lambda_1(F)=0$, then we are done since $K$ being non-empty implies $\lambda_1(K) \ge 1$. So let us assume $\lambda_1(F)>0$ from now on.
    Let $N:=V(K)\setminus V(F).$
    Let $A_F$ be the adjacency matrix of $F$ padded with zeros in all rows and columns corresponding to $N$ to make it a $V(K) \times V(K)$ matrix. Let $\textbf{x}$ be a unit eigenvector of $\lambda_1(F)$ chosen so that its restriction to $N$ is the zero vector. Moreover, by the Perron-Frobenius Theorem, we can choose \textbf{x} to have all non-negative coordinates. Let also $\textbf{x}_C$ and $\textbf{x}_L$ be restrictions of $\textbf{x}$ to $C$ and $L$ respectively (with all other entries replaced by zeros), so $\textbf{x}=\textbf{x}_C+\textbf{x}_L$. 
    
    Since $\textbf{x}$ is a unit eigenvector of $A_F$ with eigenvalue $\lambda_1(F)$, we get $\textbf{x}^{\intercal}A_F\textbf{x}=\lambda_1(F)$ and $\textbf{x}^{\intercal}A_F^{2\ell}\textbf{x}=\lambda_1(F)^{2\ell}$. Now observe that $A_F^{2\ell} \le A_K^{2\ell}-I_C$, where the inequality is meant elementwise and $I_C$ is the $V(K) \times V(K)$ matrix with diagonal entries corresponding to $C$ equal to one and all other entries being zero (so identity when restricted to $C$). This follows since the diagonal entries of $A_F^{2r}$ and $A_K^{2\ell}$ count the number of closed walks of length $2\ell$ starting at a vertex in $F$ and $K$, respectively and for a vertex in $C$, the walk taking $\ell$ steps to a vertex in $N$ and back is counted for $K$ but not for $F$. Since the entries of $\textbf{x}$ are nonnegative, we have
    $$\lambda_1(F)^{2\ell}=\textbf{x}^{\intercal}A_F^{2r}\textbf{x}\le \textbf{x}^{\intercal}(A_K^{2\ell}-I_C)\textbf{x}=\textbf{x}^{\intercal}A_K^{2\ell}\textbf{x}-\textbf{x}^{\intercal}I_C\textbf{x}\le \lambda_1(K)^{2\ell}-\textbf{x}^{\intercal}_C\textbf{x}_C.$$
    So it will suffice to show that $\textbf{x}^{\intercal}_C\textbf{x}_C=\|\textbf{x}_C\|^2 \geq \eps^2$. If $\textbf{x}_L = 0$ (or $L=\emptyset$) then we are done since $1 = \|\textbf{x}\|^2 = \|\textbf{x}_C\|^2 + \|\textbf{x}_L\|^2=\|\textbf{x}_C\|^2$ and otherwise, observe that
    \begin{align*}\lambda_1(F)=\textbf{x}^{\intercal}A_F\textbf{x}&=\textbf{x}_C^{\intercal}A_F\textbf{x}_C+2\textbf{x}_C^{\intercal}A_F\textbf{x}_L+\textbf{x}^{\intercal}_LA_F\textbf{x}_L\\
    &\le \lambda_1(F)\textbf{x}^{\intercal}_C\textbf{x}_C+2\|\textbf{x}_C\| \cdot \|A_F\textbf{x}_L\|+ \lambda_1(F[L]) \textbf{x}^{\intercal}_L\textbf{x}_L,\\
    &\le \lambda_1(F)\|\textbf{x}_C\|^2+2\|\textbf{x}_C\| \cdot \lambda_1(F)\|\textbf{x}_L\|+\lambda_1(F) \left(1 -{4\eps} \right)\|\textbf{x}_L\|^2\\
    &=\lambda_1(F)+2\lambda_1(F)\|\textbf{x}_L\|^2\left(\frac{\|\textbf{x}_C\|}{\|\textbf{x}_L\|}-{2\eps}\right),
    \end{align*}
    so that, using the fact that $\lambda_1(F) > 0$, we conclude that $\frac{\|\textbf{x}_C\|}{\sqrt{1-\|\textbf{x}_C\|^2}} = \frac{\|\textbf{x}_C\|}{\|\textbf{x}_L\|} \geq 2\eps \ge \frac{\eps}{\sqrt{1-\eps^2}}$. In the above, the first inequality follows from the variational definition of eigenvalues for the first and third term (applied for $F[L]$ for the third term) and the Cauchy-Schwarz inequality for the second term. The desired result now follows since the function $\frac{t}{\sqrt{1-t^2}}$ is increasing on $(0,1)$.
\end{proof}

The following simple observation will prove useful in several places. It states that in a graph $G$, one cannot have two disjoint sets of vertices with no edges between them such that they both induce subgraphs with the largest eigenvalue larger than $\lambda_2(G)$.
\begin{lem}\label{obs:disjoint-supports}
    Let $G$ be a connected graph and let $V,U \subseteq V(G)$ be disjoint with no edges between $V$ and $U$. Then, $\lambda_1(G[U]) < \lambda_2(G)$ or $\lambda_1(G[V]) < \lambda_2(G)$ or $\lambda_1(G[U])=\lambda_1(G[V])=\lambda_2(G)$.
\end{lem}
\begin{proof}
    Let $A$ be the adjacency matrix of $G$ and let $\textbf{z}$ be an eigenvector of $\lambda_1(G)$.
    Suppose, towards a contradiction, that both $\lambda_1(G[U]),\lambda_1(G[V]) \ge \lambda_2(G)$ and that at least one of these inequalities is strict. Let $\textbf{u}$ and $\textbf{v}$ be corresponding unit length eigenvectors of $G[U]$ and $G[V]$ respectively, padded with zeros so that $\textbf{v},\textbf{u}\in \R^{V(G)}$. Note that since $V$ and $U$ are disjoint, $\langle\textbf{v},\textbf{u}\rangle=0$. Let $a,b \in \R$ be chosen so that the vector $\textbf{x}=a\textbf{v}+b\textbf{u}$ is orthogonal to $\textbf{z}$ and is of unit length (so $a^2+b^2=1$). Via the Perron-Frobenius theorem, we may assume that all entries of $\textbf{u}$ and $\textbf{v}$ are nonnegative and since $G$ is connected, that all entries of $\textbf{z}$ are positive. It follows that neither $\textbf{v}$ nor $\textbf{u}$ can be orthogonal to $\textbf{z}$ and thus $a,b \neq 0$, so that we conclude
    $$\textbf{x}^{\intercal}A\textbf{x}=a^2 \textbf{v}^{\intercal}A\textbf{v} +b^2\textbf{u}^{\intercal}A\textbf{u}+2ab\textbf{v}^{\intercal}A\textbf{u}=a^2\lambda_1(G[V])+b^2\lambda_1(G[U])>\lambda_2(G)(a^2+b^2)=\lambda_2(G),$$
    where $\textbf{v}^{\intercal}A\textbf{u}=0$ follows from our assumption that there are no edges between vertices in the supports of $\textbf{v}$ and $\textbf{u}$. On the other hand, $\lambda_2(G)$ is the maximum of $\textbf{y}^\intercal A \textbf{y}$ over all unit length vectors $\textbf{y}$ which are orthogonal to $\textbf{z}$, so taking $\textbf{y} = \textbf{x}$ gives a contradiction.
\end{proof}

The following lemma says that in a graph $G$, either the second eigenvalue multiplicity is small or any subgraph of $G$ with largest eigenvalue larger than $\lambda_2(G)$ must be very large.   

\begin{lem}\label{lem:small-eval}
    Let $G$ be a connected graph with second eigenvalue $\lambda_2$ and maximum degree $\Delta$. If $H$ is a non-empty subgraph of $G$ with $\lambda_1(H) >  \lambda_2,$ then 
    $$m_G(\lambda_2) \le |H| \Delta.$$
\end{lem}
\begin{proof}
First, we may assume without loss of generality that $H$ is a vertex-minimal non-empty subgraph satisfying $\lambda_1(H) >  \lambda_2$. In particular, this allows us to assume $H$ is connected. This implies that there are at most $|H|+(\Delta-1)|H|=\Delta|H|$ vertices at distance at most one from $H$.

Let $H'$ be the induced subgraph of $G$ consisting only of vertices at distance at least two from $H$. This guarantees that $H$ and $H'$ are vertex disjoint and that there are no edges between them, so \Cref{obs:disjoint-supports} implies we must have $\lambda_1(H') < \lambda_2$. This implies that by removing up to $\Delta|H|$ vertices at distance at most one from $H$, we obtain a graph with the largest eigenvalue smaller than $\lambda_2$. The desired conclusion now follows by the Cauchy interlacing theorem.
\end{proof}

The following theorem is behind our new bound on the multiplicity of the second eigenvalue in the superexponential regime. 
At a high level, the theorem says that given a graph on $n$ vertices with second eigenvalue $\lambda_2$ and maximum degree $\Delta$, if $B$ is the minimum size of a ball of radius $r$ centered at a vertex with degree larger than $\lambda_2-1$ and $n$ is large enough compared to $r,B,\lambda_2$, and $\Delta$, then the second eigenvalue multiplicity is at most $\frac{n}{B-1}$. Note that the parameter $B$ here can be thought of as some kind of measure of local expansion of the graph.

The general strategy behind the proof will be to remove a few choice vertices from our graph in such a way as to ``locally'' decrease the largest eigenvalue substantially below $\lambda_2$. We then convert this local bound to a global bound on the eigenvalue multiplicity of $\lambda_2$ for the reduced graph. The Cauchy interlacing theorem then tells us that bringing back the removed vertices can only increase the multiplicity by the number of vertices we are reintroducing, thereby yielding the desired result. We note that at this high level, our argument is similar to that of \cite{yufei-paper}, with the main differences coming from how we choose the vertices to remove and how we control the impact of this on the eigenvalues (at a local level).

We start by showing that balls of some large radius $S$ around any vertex must have the largest eigenvalue at most $\lambda_2$. Next, we consider a maximal collection of disjoint large balls of radius $r$ (each of size at least $B$). We then remove the centers of these balls from the graph and mark any vertex within distance $2r$ from one of them as ``covered''. By maximality, we now know that among the uncovered vertices, there are no large balls of radius $r$ around a vertex. 
So by our assumption, all uncovered vertices must have low degree. 
This brings us to the setting of \Cref{lem:partial-net} which guarantees that by removing our center vertices, we can decrease the largest eigenvalue of any ball of radius $S$ substantially below $\lambda_2$.

\begin{thm}\label{thm:multiplicity-approx-bound}
    Let $G$ be a connected $n$-vertex graph with second eigenvalue $\lambda_2\ge 1$ and $\Delta\ge \Delta(G)$. Suppose that for some integer $r \ge1$, any vertex of $G$ with degree larger than $\lambda_2-1$ satisfies $|B_G^{(r)}(v)|\ge B \ge 1$. Then, provided $2^{9}r\lambda_2^{4r+2} \log B \le \log_{\Delta} n $, we have
    $$m_G(\lambda_2)\le \frac{n}{B-1}.$$ 
\end{thm}

\begin{proof}
The desired bound is trivial unless $B>2$. This in turn implies $\Delta \ge 2$ as otherwise $|B_G^{(r)}(v)|\le 2$ for any $r$ and $v$. Note also that since $G$ is connected, the Perron-Frobenius theorem implies that $\lambda_1(G)>\lambda_2.$

Let $S$ be the minimum integer for which $\lambda_1(B_G^{(S+1)}(u))>\lambda_2$ for some $u \in V(G)$ (note that such an $S\ge 0$ exists since $\lambda_1(G)>\lambda_2$ and for large enough radius $B_G^{(S+1)}(u)=G$). In particular, by our choice of $S$, for any $v \in V(G)$ we have $\lambda_1(B_G^{(S)}(v))\le \lambda_2$. 
\Cref{lem:small-eval} implies that $m_G(\lambda_2) \le |B_G^{(S+1)}(u)|\Delta$ and so, we are done unless 
\begin{equation}\label{eq:s-lwr-bnd}
\frac{n}{B-1}<|B_G^{(S+1)}(u)|\Delta \le \Delta^{S+1}+\Delta^{S+2}\le \Delta^{S+3}.
\end{equation}

Let $\mathcal{B}$ be a maximal collection of disjoint balls in $G$ having radius $r$ and size at least $B$ and let $N$ denote the set of their center vertices. Observe that since each ball has size at least $B$, we have
\begin{equation}\label{eq:N_2-upr-bnd}
    |N| \le \frac{n}{B}.
\end{equation}
Let $C'$ consist of all vertices at distance at most $2r$ from a vertex in $N$ and let $L'=V(G) \setminus C'$.  
Note that any $v\in L'$ is at a distance of at least $r+1$ from any vertex belonging to a ball in $\mathcal{B}$. Therefore, $B_G^r(v)$ is disjoint from all balls in $\mathcal{B}$ and by maximality of $\mathcal{B}$, it must have size less than $B$. By our assumption, this implies that $d_G(v) \le \lambda_2-1$.

Let $R=S-2r$ and let $H$ be the graph obtained by deleting vertices in $N$ from $G$. 
By counting closed walks of length $2R$ in $H$, we get 
    \begin{equation}\label{eq:walks}
    m_{H}(\lambda_2)\cdot \lambda_2^{2R} \le \sum_{i=1}^{|H|} \left(\lambda_i\left(H\right)\right)^{2R}=\tr A_{H}^{2R} =\sum_{v \in V\left(H\right)}\textbf{e}_v^{\intercal} \: A_{B_{H}^{(R)}(v)}^{2R}\:\textbf{e}_v 
    \le \sum_{v \in V\left(H\right)} \left(\lambda_1\left(B_{H}^{(R)}(v)\right)\right)^{2R},
    \end{equation}
since $\textbf{e}_v^{\intercal} \: A_{B_{H}^{(R)}(v)}^{2R}\:\textbf{e}_v$ is the diagonal entry of $A_{B_{H}^{(R)}(v)}^{2R}$ corresponding to $v$ and hence equals the number of closed walks of length $2R$ starting and ending at $v$ in $H$ (since any such walk must be contained within $B_{H}^{(R)}(v)$).

    Now let $v \in V(H)$ and $F:=B_{H}^{(R)}(v)$. In view of \eqref{eq:walks}, observe that to bound the multiplicity of $\lambda_2$ in $H$, it will suffice for us to obtain an upper bound on $\lambda_1\left(F\right)$ in terms of $\lambda_2$, as follows.
    
    \begin{claim*}
        If we set $\eps=\frac{1}{8\lambda_2}$, we have $$\lambda_1(F)^{4r} \le \lambda_2^{4r}-\eps^2.$$
    \end{claim*}
    \begin{cla_proof}
    First note that $\lambda_1(F) \leq \lambda_1\left( B^{(S)}_G(v) \right) \leq \lambda_2$. If $\lambda_1\left(F\right) < \lambda_2-1/2$, the desired bound easily holds, so we may assume that $\lambda_1\left(F\right) \ge \lambda_2-1/2.$ We next apply
    \Cref{lem:partial-net} 
    with $\ell=2r$, $\eps=1/(8\lambda_2)$, $K=B_G^{(S)}(v)$, $F = B^{(R)}_H(v)$, $C=C' \cap V\left(F\right)$, and $L=L' \cap V\left(F\right)$. Indeed, we can do this since every vertex $u \in C$ is at a distance of at most $2r$ to some vertex $w \in N$, so that $w \in  V\left (K \right) \setminus V\left(F\right) $ and also since every vertex of $L'$ has a degree of at most $\lambda_2-1\le \lambda_1(F)-1/2\le \lambda_1(F) \left(1 -4\eps \right)$, so in particular the same bound holds for $\lambda_1(F\left[L\right])$. 
    This implies 
     $$\lambda_1(F)^{4r}=\left(\lambda_1\left(B_{H}^{(R)}(v)\right)\right)^{4r} \le \left(\lambda_1\left(B_G^{(S)}(v)\right)\right)^{4r}-\eps^2\le \lambda_2^{4r}-\eps^2,$$
     as desired.\end{cla_proof}

    Using the claim to continue the estimate in \eqref{eq:walks}, we get 
    \begin{equation*}
    m_{H}(\lambda_2)\cdot \lambda_2^{2R}
    \le \sum_{v \in V(H)} \left(\lambda_1\left(B_{H}^{(R)}(v)\right)\right)^{2R}\le n \cdot \left(\lambda_2^{4r}-\eps^2\right)^{R/(2r)},
    \end{equation*}
    which in turn implies 
    \begin{align*}
        m_{H}(\lambda_2)&\le n \cdot \left(1-\frac{\eps^2}{\lambda_2^{4r}}\right)^{R/(2r)}
        \le n \cdot e^{-\eps^2\lambda_2^{-4r}\cdot R/(2r)}=n \cdot e^{-\lambda_2^{-4r-2}\cdot R/(2^7 r)}\le n \cdot e^{-2 \log B}\le \frac{n}{B^2},
    \end{align*}
    where in the penultimate inequality we used \eqref{eq:s-lwr-bnd},   which gives that $R=S-2r \ge \log_{\Delta} n-\log_{\Delta} (B-1)-3-2r \ge 2^8 r\lambda_2^{4r+2}\log B$.
    
    Now the Cauchy interlacing theorem combined with \eqref{eq:N_2-upr-bnd} implies 
     $$m_G(\lambda_2) \le |N|+m_H(\lambda_2)\le \frac{n}{B} + \frac{n}{B^2}\le \frac{n}{B-1},$$
    as desired.
    \end{proof}

We note that with a slight change in the above argument, we could have counted the number of eigenvalues (counted with multiplicities) in the interval $\left[\left(1-\frac{\log B}{4\log_{\Delta} n}\right)\lambda_2,\lambda_2\right].$ Indeed, the only significant change is that if \eqref{eq:s-lwr-bnd} fails to hold, then we are not done by the Cauchy interlacing theorem. Instead, if there is a $u$ such that $B_G^{(S+1)}(u)$ is small, while its largest eigenvalue is larger than $\lambda_2$, we can simply add $B_G^{(S+1+r)}(u)$ into $N$ in the above argument and continue as before. Another slight change is in \eqref{eq:walks}, where instead of $\lambda_2^{2R}$, we get as a lower bound the lower limit of the interval at hand, which determines how wide the interval we can take.

We also remark that in the above theorem, the assumption $\lambda_2 \ge 1$ is more of a technicality. Indeed, if $\lambda_2<1$ the graph cannot contain an induced matching of size $2$ by \Cref{obs:disjoint-supports}. This implies that the diameter of the graph is at most 3 and in particular that the number of vertices is at most $\Delta^3+1$.

The following corollary is a more precise version of \Cref{cor:multiplicity-bound}. It is a simple consequence of the previous theorem combined with the tiny amount of expansion we obtain from the connectivity assumption.
\begin{cor}\label{cor:multiplicity-bound-precise}
    Let $G$ be a connected $n$-vertex graph with second eigenvalue $\lambda_2,$ $\Delta=\Delta(G)\ge 2$, and $\delta=\delta(G)$.  Then, provided $\log n \ge \Delta^{48}$, we have
    $$m_G\left(\lambda_2\right)\le \frac{n}{\lambda_2+\delta \log_{\Delta} \log n/42   }.$$ 
\end{cor}
\begin{proof}
First, note that we must have $\lambda_2 \ge 1$. Indeed, otherwise, $G$ cannot contain an induced matching of size $2$ by \Cref{obs:disjoint-supports}, which implies that the diameter of $G$ is at most 3 and, in particular, that the number of vertices in $G$ is at most $(\Delta+1)^3  < 2^{\Delta^{48}}$.

Let $r=\floor{\frac16\log_{\Delta} \log{n}}\ge 8$ and $B=\lambda_2+\delta r/6+1$ so that $\log B \le \lambda_2+\delta r/6\le \Delta r/3$. This implies that
$$2^8 r\Delta^{4r+2}  \log \Delta \log B\le \frac{2^6}{3^3}(\log \log n)^2\Delta^{4r+3}\le 3 (\log \log n)^2 (\log n)^{2/3}\Delta^{3} \le\log n.$$

Now observe that in a connected graph with minimum degree $\delta$, we have $\left|B_G^{(r)}(v)\right|\ge \min\{d_G(v)+1+\floor{\frac{r-1}{3}}(\delta+1),n\}$. Indeed, if $B_G^{(r)}(v)\setminus B_G^{(r-1)}(v)=\emptyset,$ then $\left|B_G^{(r)}(v)\right|= n$. Otherwise, there exists a vertex $v_r$ at distance $r$ from $v$. Let $v=v_0,v_1,\ldots,v_r$ be vertices making a path from $v$ to $v_r$. Since $v_r$ is at a distance of precisely $r$ from $v$, we know that any vertex $v_{i}$ together with its neighborhood $N_G(v_i)$ is restricted to $B_G^{(i+1)}(v) \setminus B_G^{(i-2)}(v)$. It follows that the sets $N_G(v_{3i}) \cup \{v_{3i}\}$ for $i = 0, 1, \ldots, \floor{\frac{r-1}{3}}$ are all disjoint and lie in $B_G^{(r)}(v)$. Since $|N_G(v_{3i}) \cup \{v_{3i}\}| = d_G(v_{3i}) + 1 \geq \delta + 1$, we conclude that $\left|B_G^{(r)}(v)\right| \geq d_G(v) + 1 + \floor{\frac{r-1}{3}}(\delta+1)$.

So for any vertex $v$ with degree larger than $\lambda_2-1$, we have $\left|B_G^{(r)}(v)\right|\ge \lambda_2+r\delta/6+1=B$ and therefore, we may apply \Cref{thm:multiplicity-approx-bound} to obtain 
$$m_G\left(\lambda_2\right)\le \frac{n}{\lambda_2+\delta r/6 }\le \frac{n}{\lambda_2+\delta \log_{\Delta} \log n/42},$$
as desired.
\end{proof}

As with \Cref{thm:multiplicity-approx-bound}, the above corollary can also be generalized to obtain the same upper bound when counting eigenvalues in the interval $\left[\left(1-\Omega\left(\frac{\log \log \log n}{\log n}\right)\right)\lambda_2,\lambda_2\right]$. In this form, the corollary is essentially tight thanks to a construction from \cite{schildkraut-paper}, which exhibits a graph with maximum degree $6$ having at least $\Omega(n/\log \log n)$ eigenvalues in a slightly smaller interval.

We proceed with our result in the denser regime. We state and prove it below across two regimes to showcase the transition from the upper bound being approximately $\frac{n}{\log_{\Delta} n}$ to being approximately $\frac{n}{\lambda_2}$. The transition occurs around $ \lambda_2 \approx \log_{\Delta} n$. 
We note that the main part of the following theorem is the first regime, with the remaining being included to bridge the gap between it and \Cref{cor:multiplicity-bound-precise}. 

\begin{thm}\label{thm:multiplicity-bound-dense}
    Let $G$ be a connected $n$-vertex graph with $\lambda_2=\lambda_2(G)> 0$ and $\Delta\ge \Delta(G)\ge 2$. Then,
  $$m_G(\lambda_2)\le 
    \begin{cases}\displaystyle
        \frac{4n \log^2 \log_{\Delta} n}{\log_{\Delta} n}
 &\text{if } \lambda_2 \log \lambda_2 \ge \log_{\Delta} n > 2, \vspace{0.2cm}\\ \displaystyle
    \frac{4n \log (\lambda_2+1)}{\lambda_2}
 &\text{if }  \lambda_2 \log \lambda_2 \le \log_{\Delta} n 
    \end{cases}$$ 
\end{thm}
\begin{proof}
Note that in the first regime, we may assume $\log_{\Delta} n \ge  4\log^2\log_{\Delta} n$ or the desired inequality becomes trivial. Since $\log_\Delta n >2$ this implies $\log_{\Delta} n \ge 16.$ In the second regime, we may assume $\lambda_2 \ge 4 \log (\lambda_2+1)$, which gives $\lambda_2 \ge 16$ and $\log_{\Delta} n \ge \lambda_2 \ge 16$.

Let $S = \floor{\log_{\Delta} n}-2\ge 14$. Suppose first that for some $v \in V(G)$, we have $\lambda_1(B_G^{(S)}(v))> \lambda_2$.
Using \Cref{lem:small-eval} with $H=B_G^{(S)}(v)$ (note that $S \ge 1$ implies $H$ is non-empty), we obtain
$$m_G(\lambda_2) \le \left|B_G^{(S)}(v)\right|\Delta\le \Delta^{S+1}\le \frac{n}{\Delta} \le \frac{n}{\lambda_2},$$
where we used $\Delta \ge 2$ and $S \ge 3$ in the second inequality and the definition of $S$ in the third. This bound clearly suffices if we are in the second regime and by our upper bound assumption on $\log_{\Delta} n$ in the first regime. So, we may now assume that for any $v \in V(G)$, we have 
\begin{equation}\label{eq:ball-eval-upr-bnd}
\lambda_1(B_G^{(S)}(v))\le \lambda_2.
\end{equation}

Our next goal is to count the closed walks of length $2S$ in a subgraph of $G$ that we will obtain by deleting a few vertices. Let us first observe that the total number of such walks in $G$ (counted with a specified starting vertex) equals
\begin{equation}\label{eq:upr-bnd-walks-G}
    \sum_v\textbf{e}_v^{\intercal} A_G^{2S}\textbf{e}_v \le \sum_{v} \lambda_1\left(B_G^{(S)}(v)\right)^{2S}\le n \lambda_2^{2S},
\end{equation}
where in the final inequality, we used \eqref{eq:ball-eval-upr-bnd}.

We now separate walks in $G$ according to whether their support size is bigger or smaller than some parameter $\ell \le S$ to be chosen later, depending on the regime we are in.
    Note that given a vertex $v$ in $G$, the number of closed walks of length $2\ell$ starting at $v$ equals 
    $$\textbf{e}_v^{\intercal} A_G^{2\ell} \textbf{e}_v=\textbf{e}_v^{\intercal} A_{B^{(\ell)}_G(v)}^{2\ell} \textbf{e}_v \le \lambda_1\left(B^{(\ell)}_G(v)\right)^{2\ell}\le \lambda_2^{2\ell},$$
    where we used \eqref{eq:ball-eval-upr-bnd} once again.
    This implies that there can be at most $\lambda_2^{2\ell}$ connected induced subgraphs of $G$ on $\ell$ vertices containing $v$, since each such graph can be identified with a unique walk of length $2\ell$ (for example, to the DFS traversal of one of its spanning trees). Moreover, once we fix such a graph, the number of closed walks of length $2S$ starting at $v$ restricted to this subgraph is at most $\ell^{2S}$. So putting the bounds together, we conclude there can be at most 
    \begin{equation}\label{eq:short-walks-bound}
        n \cdot \lambda_2^{2\ell} \cdot \ell^{2S} 
    \end{equation}
    closed walks of length $2S$ in $G$ (counted with specified starting vertex) with support of size at most $\ell$.

    If we take a uniformly random subset consisting of $t=\ceil{n\cdot \frac{\log \ell}{\ell}}$ vertices of $G$, for any walk with support at least $\ell$, the probability that we do not sample any of its vertices is at most $$\binom{n-\ell}{t} \Big / \binom{n}{t} \le (1-\ell/n)^{t}\le 2^{-t\ell/n}\le \frac{1}{\ell}.$$

    This, combined with \eqref{eq:upr-bnd-walks-G}, implies that there is an outcome in which, by removing $t$ vertices from $G$, we obtain a graph $H$ which has at most $n  \lambda_2^{2S}/\ell$ closed walks of length $2S$ (counted with specified starting vertex) with support size at least $\ell$. Together with $\eqref{eq:short-walks-bound}$, it follows that the total number of closed walks of length $2S$ in $H$ is at most $n \lambda_2^{2\ell} \ell^{2S} + n \lambda_2^{2S}/\ell$ and thus,
    $$n \lambda_2^{2\ell} \ell^{2S} + n \lambda_2^{2S}/\ell \ge \sum_{v \in H} \textbf{e}_v^{\intercal} A_H^{2S}\textbf{e}_v = \tr \left(A_H^{2S}\right)\geq m_H(\lambda_2) \cdot \lambda_2^{2S}.$$
    Combining this with the Cauchy interlacing theorem gives an upper bound of 
    \begin{equation}\label{eq:multiplicity-all-regimes}
        m_G(\lambda_2) \le t+ m_H(\lambda_2)\le n \cdot \frac{\log \ell}{\ell} + 1 + n \cdot \lambda_2^{2\ell-2S}\ell^{2S}+\frac{n}{\ell}\le n \cdot \frac{2+\log \ell}{\ell} + n \cdot \left(\frac{\ell}{\lambda_2^{1-\ell/S}}\right)^{2S},
    \end{equation}
    where we used $\ell \le S \le \log_{\Delta} n \le n$.

    We now make separate choices for $\ell$ depending on the regime. 
    If  $\lambda_2 \log \lambda_2 \ge \log_{\Delta} n > S$, we set $\ell:=  \floor{\frac{S}{2\log S}}\ge \frac{S}{4\log S}.$ Note that since $\ell\le \frac{S}{2\log S}$ and $\lambda_2 \ge \frac{S}{\log S}$, we get 
    $$\left(\frac{\ell}{\lambda_2^{1-\ell/S}}\right)^{2S}\le \left(\frac{S/(2\log S)}{(S/\log S)^{1-\ell/S}}\right)^{2S}\le 2^{-S},$$
    where we used in the first inequality that the left-hand side is decreasing in $\lambda_2$. Plugging this into \eqref{eq:multiplicity-all-regimes}, we get   
    $$ m_G(\lambda_2) \le n \cdot \frac{2+\log S-\log (2\log S)}{ S/ (3\log S)} + n \cdot 2^{-S}\le n \cdot \frac{3\log^2 S}{ S}  \le n \cdot \frac{4\log^2 \log_{\Delta} n}{\log_{\Delta} n},
    $$
    where we used $S \ge \log_{\Delta} n -3 \ge 13.$
    
    Finally, if $\log_{\Delta} n \ge \lambda_2 \log \lambda_2$, we set $\ell:=\lambda_2/3$ to get from \eqref{eq:multiplicity-all-regimes} 
    $$ m_G(\lambda_2) \le 3n \cdot \frac{1/2+\log \lambda_2}{\lambda_2}+ n \cdot 2^{-2S}\le \frac{4n\log \lambda_2}{\lambda_2},$$
    where we use $\lambda_2 \ge 16$ in the first inequality to bound the second term via $\frac{\lambda_2 \log \lambda_2}{3S} \le \frac{\lambda_2 \log \lambda_2}{3(\lambda_2 \log \lambda_2-3)}\le \log(3/2)$ and in the second inequality, we use that $2S \ge 2\log_{\Delta} n -6 \ge 2\lambda_2 \log \lambda_2-6 \ge \log \lambda_2$.
    \end{proof}

We note that by being a bit more careful, one can remove one of the $\log \log_{\Delta}n$ terms in the first regime of the above result when $\log_{\Delta} n \le \lambda_2^{1-\eps}$ for any fixed $\eps>0$. We decided not to do so to keep the argument as simple and non-technical as feasible.

We now put together previous results to conclude \Cref{thm:multiplicity-main}.

\begin{thm}\label{thm:multiplicity-combined}
    Let $G$ be an $n$-vertex graph with second eigenvalue $\lambda_2$ and $\Delta \ge \Delta(G)\ge 2.$  Then, 
    $$m_G\left(\lambda_2\right)\le 5n \cdot \max\left\{\frac{\log (\lambda_2+1)}{\lambda_2},\frac{\log^2 (1+\log_{\Delta} n)}{\log_{\Delta} n}\right\}.$$
    Moreover, for any $0<\eps\le 1$ if $\log_{\Delta} n \ge \eps^{-1}(2\lambda_2+2)^{14} \log (\lambda_2+2)$, then 
     $$m_G\left(\lambda_2\right)\le \frac{n}{\lambda_2+1}+n^{\eps}.$$
\end{thm}

\begin{proof}
First, we claim that we may assume that $\log_{\Delta} n \ge 10$.
Indeed, for the first part of the theorem, this follows since the trivial upper bound of $n$ is smaller than our second term unless $\log_{\Delta} n  > 5 \log^2 (1+\log_{\Delta} n)$ and for the moreover part, this follows from the lower bound assumption on $\log_{\Delta} n$.
This implies that $\lambda_2 \ge 0$ (since it implies that our graph can't be complete). This in turn ensures the first term in the maximum is always larger than $\frac{1}{\lambda_2+1}$ and shows that if we establish an upper bound of $\frac{n}{\lambda_2+1}$ we are done in both the main and moreover parts of the theorem. 
  Furthermore, we may now assume that $\lambda_2 \ge 1$. Indeed, if $\lambda_2 < 1$, then by the Cauchy interlacing theorem, $G$ cannot contain an induced matching of size $2$, which implies that the diameter of $G$ is at most $3$ and, in particular, $m_{G}(\lambda_2) \le n\le \Delta^3+1 \le \frac{n}{\Delta+1} < \frac{n}{\lambda_2+1}$.

Suppose first that $\lambda_1(G)=\lambda_2$. Then, $m_G(\lambda_2)$ equals the number of connected components of $G$ with largest eigenvalue equal to $\lambda_2$. Since any such graph must have at least $\lambda_2+1$ vertices, we get $m_G(\lambda_2) \le \frac{n}{\lambda_2+1},$ as desired.
 
Let us now assume $\lambda_1(G) > \lambda_2$. Let $G'$ denote the connected component of $G$ with the largest eigenvalue equal to $\lambda_1(G)$. Note that the remaining components must have largest eigenvalues at most $\lambda_2$, so if $\lambda_2(G')<\lambda_2$ we again get $\frac{n}{\lambda_2+1}$ as an upper bound. So we may assume $\lambda_2(G')=\lambda_2$ and we get
\begin{equation}\label{eq:GtoG'}
    m_G(\lambda_2)\le m_{G'}(\lambda_2) + \frac{n-n'}{\lambda_2+1},
\end{equation}
where $n'\le n$ denotes $|V(G')|$. Note that $\Delta(G)\ge 2$ implies $1<\lambda_1(G)=\lambda_1(G') \le \Delta(G')$, so $\Delta(G') \ge 2$.

%We claim we may assume that $\log_{\Delta} n \ge 10$. Indeed, for the first part of the theorem, this follows since the trivial upper bound of $n$ is smaller than our second term unless $\log_{\Delta} n  > 4 \log^2 (1+\log_{\Delta} n)$. For the moreover part, this follows from the lower bound assumption on $\log_{\Delta} n$ in terms of $\lambda_2$, since the desired bound is trivial unless $\lambda_2 >0$. Furthermore, we may assume that $\lambda_2 \ge 1$. Indeed, if $\lambda_2 < 1$, then $G'$ cannot contain an induced matching of size $2$ by \Cref{obs:disjoint-supports}, which implies that the diameter of $G'$ is at most 3 and, in particular, $m_{G'}(\lambda_2) \le n'\le \Delta^3+1 \le \frac{n}{\Delta+1} < \frac{n}{\lambda_2+1}$. 

To verify the first part of the theorem, note that since $\log_{\Delta} n \ge 10$, we may assume $\log_{\Delta} n' >8$, as otherwise $m_G(\lambda_2) \le n'+\frac{n}{\lambda_2+1}\le \Delta^8+\frac{n}{\lambda_2+1}\le \frac{2n}{\lambda_2+1}.$
We can now apply \Cref{thm:multiplicity-bound-dense} to $G'$ (with $\Delta\ge \Delta(G) \ge \Delta(G')\ge 2$). The second regime gives us an upper bound of $\frac{4n' \log (\lambda_2+1)}{\lambda_2}\le \frac{4n \log (\lambda_2+1)}{\lambda_2}$, as desired. The first regime gives an upper bound of
$$ \frac{4n'\log^2 \log_{\Delta} n'}{\log_{\Delta} n'}\le \frac{4n'\log^2 \log_{\Delta} n}{\log_{\Delta} n'}\le \frac{4n\log^2 \log_{\Delta} n}{\log_{\Delta} n},$$
where we used $\frac{n'}{\log n'} \le \frac{n}{\log n}$ in the final inequality. Combining with \eqref{eq:GtoG'} in both cases gives the desired bound for $G$.

For the moreover part, since by \eqref{eq:GtoG'} $m_G(\lambda_2)\le n' + \frac{n}{\lambda_2+1}$ we are done unless $n' \ge n^{\eps}$. This gives $\log_{\Delta} n' >\eps \log_{\Delta} n \ge 2^9 \cdot 3 \cdot \lambda_2^{14} \log (\lambda_2+2)$ and we may apply \Cref{thm:multiplicity-approx-bound} with $r=3$ and $B=\lambda_2+2$ to $G'$ to get $m_{G'}(\lambda_2) \le \frac{n'}{\lambda_2+1}$. Combining with \eqref{eq:GtoG'} gives the desired bound for $G$.
\end{proof}

We note that by choosing $\eps$ optimally, the error term $n^{\eps}$ is of the form $\Delta^{\lambda_2^{O(1)}}$.

\section{Improved bounds for equiangular lines}\label{sec:equiangular}
In this section, we prove \Cref{thm:main}. We begin with some setup and two lemmas that we will need.

For any collection of $n$ $\alpha$-equiangular lines in $\R^r$, choosing a unit vector along each line yields a collection of $n$ unit vectors $\Ce$ having the property that $\inprod{\textbf{u}}{\textbf{v}} \in \{\pm \alpha\}$ for all distinct $\textbf{u}, \textbf{v} \in \Ce$; we refer to such a collection $\Ce$ as a \emph{spherical $\{\pm \alpha\}$-code}. We define its corresponding Gram matrix $M \colon \Ce \times \Ce \to \R$ by $M(\textbf{u}, \textbf{v}) = \inprod{\textbf{u}}{\textbf{v}}$ and its \emph{corresponding graph} $G$ as having vertex set $V(G) = \Ce$ and edges made by pairs with negative inner products, i.e.\ $E(G) = \{\textbf{u}\textbf{v} : \inprod{\textbf{u}}{\textbf{v}} = -\alpha \}$. Now let $A = A_G$ be the adjacency matrix of the corresponding graph $G$. Then, $M$ and $A$ satisfy
\begin{equation} \label{eq_gram_adjacency}
    M = (1-\alpha) \: I_n + \alpha J_n - 2 \alpha A,
\end{equation}
where $I_n$ is the $n \times n$ identity matrix and $J_n$ is the $n \times n$ matrix with all entries equal to one. 
A key, standard property of any Gram matrix that we will use is that it is positive semidefinite, i.e.\ $\textbf{x}^{\intercal} M \textbf{x}\ge 0$ for any $\textbf{x}\in \R^{\Ce}$. 

For any graph $G$ with adjacency matrix $A$, we define the parameter $$\beta(G) := \max_{\substack{\textbf{x} \in \R^{V(G)} \backslash \{\textbf{0}\} \\ \textbf{x} \perp \one}} \frac{\textbf{x}^{\intercal} A\textbf{x}}{\textbf{x}^{\intercal} \textbf{x}}.$$
$\beta(G)$ plays a similar role in general graphs as the second eigenvalue does in the regular case. Indeed, by the variational definition of eigenvalues, $\beta(G) \geq \lambda_2(G)$ with equality when $G$ is regular. The following lemma, which appears in an almost identical form in \cite{igor-paper}, gives us an upper bound on this parameter for the corresponding graph of a spherical $\{\pm \alpha\}$-code, in terms of $\alpha$. 

\begin{lem} \label{lem_beta_upper}
Let $0 < \alpha < 1$ and let $\Ce$ be a spherical $\{\pm \alpha\}$-code in $\R^r$ with corresponding graph $G$. Then $$\beta(G) \leq \frac{1-\alpha}{2\alpha}$$ and if $|\Ce| \geq r + 2$, then we have equality above and $\beta(G)$ is an eigenvalue of $G$.
\end{lem}
\begin{proof}
Let $M$ be the Gram matrix of $\Ce$, let $A = A_G$ be the adjacency matrix of $G$, and let $\textbf{x} \in \R^\Ce \backslash \{0\}$ be such that $\textbf{x} \perp \one$. Then using \eqref{eq_gram_adjacency} and the fact that $M$ is positive semidefinite, we have
$$
0 \leq \textbf{x}^{\intercal} M \textbf{x} 
= (1-\alpha) \textbf{x}^{\intercal} I_n \textbf{x} - 2 \alpha \textbf{x}^{\intercal} A \textbf{x}= (1-\alpha) \textbf{x}^{\intercal} \textbf{x} - 2 \alpha \textbf{x}^{\intercal} A \textbf{x}\implies \frac {\textbf{x}^{\intercal} A \textbf{x}}{\textbf{x}^{\intercal} \textbf{x}} \le \frac{1-\alpha}{2\alpha},
$$
from which the desired upper bound on $\beta(G)$ follows. 

Moreover, if $|\Ce| \geq r+2,$ then since $\rk(M) \leq r$, the nullspace of $M$ has dimension at least 2 and must therefore contain a nonzero vector $\textbf{y} \perp \one$. Again using \eqref{eq_gram_adjacency}, we have that $A \textbf{y} = \frac{1-\alpha}{2\alpha} \textbf{y}$ and $\beta(G) \geq \frac{\textbf{y}^\intercal A \textbf{y}}{\textbf{y}^\intercal \textbf{y}} = \frac{1-\alpha}{2\alpha}$, which implies that $\beta(G)$ equals $\frac{1-\alpha}{2\alpha}$ and is an eigenvalue of $G$.
\end{proof}

Note that when choosing the unit vectors making the spherical code corresponding to a given collection of equiangular lines, we have two choices for each of their directions. The operation of changing the direction for one such vector from $\textbf{v}$ to $-\textbf{v}$ is called a \emph{switching} and it corresponds to 
swapping the set of neighbors and non-neighbors of $\textbf{v}$ in the corresponding graph. 
It was shown in \cite{igor-paper} that there is always a switching for which the corresponding graph has small maximum degree. We prove a weaker version of this lemma with a different but simpler proof, since it will suffice for our applications. We will make use of the Frobenius inner product of two matrices with the same dimensions defined by $\inprod{A}{B}_F=\tr(A^{\intercal}B)$. The corresponding norm is then defined as $\|A\|^2=\inprod{A}{B}_F$. We will repeatedly use the fact that if $\textbf{x},\textbf{z}, \textbf{y},\textbf{w} \in \R^r$, then $\inprod{\textbf{x}\textbf{y}^{\intercal}}{\textbf{z}\textbf{w}^{\intercal}}_F=\tr(\textbf{y}\textbf{x}^{\intercal}\textbf{z}\textbf{w}^{\intercal})=\tr(\textbf{w}^{\intercal}\textbf{y}\textbf{x}^{\intercal}\textbf{z})=\textbf{w}^{\intercal}\textbf{y}\textbf{x}^{\intercal}\textbf{z}=\inprod{\textbf{x}}{\textbf{z}}\inprod{\textbf{y}}{\textbf{w}}.$

\begin{lem}\label{lem:max-deg-upr-bnd}
    Let $0<\alpha<1$. For any family of $\alpha$-equiangular lines, there exists a choice of a unit vector along each line such that the resulting spherical $\{\pm\alpha\}$-code has a corresponding graph with maximum degree at most $6 / \alpha^4$.
\end{lem}
\begin{proof}
For any family of $\alpha$-equiangular lines, we may fix one unit vector $\textbf{w}$ along a fixed line, and for every other line, choose a unit vector $\textbf{v}$ lieing in it such that $\inprod{\textbf{v}}{\textbf{w}} = \alpha$. Let $\Ce$ be the spherical $\{\pm\alpha\}$-code consisting of all chosen vectors except for $\textbf{w}$. It will suffice for us to show that every vertex in the graph corresponding to $\Ce$ has degree either at most $1 / \alpha^4$ or at least $|\Ce| - 1 - 1/\alpha^4$ and that the number of vertices with degree at least $|\Ce| -1- 1/\alpha^4$ is at most $4/\alpha^4$. Indeed, we may then apply a switching to all vectors which were in the latter group (negating the vectors) in order to arrive at a spherical $\{\pm\alpha\}$-code which also corresponds to the given family of lines and for which the corresponding graph has maximum degree at most $5/\alpha^4+1$.

Fix $\textbf{u} \in \Ce$ and for any nonempty $X \subseteq \Ce$, define $A_X = \frac{1}{|X|}\sum_{\textbf{v}\in X}{\textbf{v}\textbf{v}^\intercal}$ and $z_X = \frac{1}{|X|}\sum_{\textbf{v}\in X}{\inprod{\textbf{v}}{\textbf{u}}}$. Observe that $||A_X||_F^2 = \frac{1}{|X|^2}\left(|X| + \alpha^2|X|(|X|-1)\right) < \alpha^2 + \frac{1}{|X|}$ and $\inprod{\textbf{u} \textbf{w}^\intercal}{A_X}_F = \frac{1}{|X|}\sum_{\textbf{v}\in X}{\inprod{\textbf{w} \textbf{u}^\intercal}{\textbf{v}\textbf{v}^\intercal}_F} = \alpha z_X$. Furthermore, if we let $Y \subseteq \Ce$ be nonempty and disjoint from $X$, then $\inprod{A_X}{A_Y}_F = \alpha^2$, so via the Cauchy-Schwarz inequality, we obtain
    $$
        \alpha |z_X - z_Y| = |\inprod{\textbf{u} \textbf{w}^\intercal}{A_X - A_Y}_F| \leq ||\textbf{u}\textbf{w}^\intercal||_F ||A_X - A_Y||_F
        = \sqrt{||A_X||_F^2 + ||A_Y||_F^2 - 2 \inprod{A_X}{A_Y}_F}
        \leq \sqrt{\frac{1}{|X|} + \frac{1}{|Y|}}.
    $$
Suppose towards a contradiction that $1/\alpha^4<d(\textbf{u})< |\Ce|-1-1/\alpha^4$. The above inequality with $X = N(\textbf{u})$ and $Y = \Ce \backslash (N(\textbf{u})\cup \{\textbf{u}\})$ implies that $4\alpha^4 \leq \frac{1}{d(\textbf{u})} + \frac{1}{|\Ce|-1 - d(\textbf{u})}<2\alpha^4$, a contradiction.
It now remains to show that $T := \left\{\textbf{u} \in \Ce : |\Ce| -1- d(\textbf{u}) \leq 1/\alpha^4 \right\}$ satisfies $|T| \leq 4/\alpha^4$. To this end, we note that $d_{G[T]}(\textbf{u}) \geq d(\textbf{u}) - |\Ce| + |T| \geq |T| -1- 1/\alpha^4$ and thus,
    \begin{align*}
        0 \leq \Big\|  \sum_{\textbf{u} \in T}{\textbf{u}}  \Big\|^2 
        = |T| + \alpha |T|(|T|-1) - 2 \alpha \sum_{\textbf{u} \in T}{d_{G[T]}(\textbf{u})}
        &\leq |T| + \alpha |T|(|T|-1) - 2 \alpha |T| \left(|T| - 1 - \frac{1}{\alpha^4} \right)\\
        &\leq |T| \left(\frac{4}{\alpha^3} - \alpha |T| \right),
    \end{align*}
giving the desired bound on $|T|$. \end{proof}

Let $\MC(n,\lambda_2,\Delta)$ denote the maximum multiplicity of the second largest eigenvalue $\lambda_2$ in a connected graph with at most $n$ vertices with maximum degree at most $\Delta$. The following lemma captures the relation between the equiangular lines problem and the eigenvalue multiplicity questions discussed in the previous section.
\begin{lem}\label{lem:lines-to-multiplicity} 
Let $0< \alpha <1 $ and $r\ge 2$ be an integer. If there exist $n$  $\alpha$-equiangular lines in $\mathbb{R}^r$, then
$$n \le r-1+\max\left\{(r-1)\cdot\frac{2\alpha}{1-\alpha},\: \MC\left(n,\frac{1-\alpha}{2\alpha},\frac{6}{\alpha^4}\right)+2\right\}.$$
Moreover, if there does not exist a graph with the largest eigenvalue equal to $\frac{1-\alpha}{2\alpha}$, then $n \leq r + 1 + \MC\left(n,\frac{1-\alpha}{2\alpha},\frac{6}{\alpha^4}\right).$
\end{lem}
\begin{proof}
    Let $\Ce$ be the spherical $\{\pm \alpha\}$-code in $\R^r$ with $|\Ce|=n$ chosen via \Cref{lem:max-deg-upr-bnd} so that its corresponding graph $G$ has maximum degree at most $6/\alpha^4$.
Let $A = A_G$ be its adjacency matrix and let $\lambda_1 \ge \lambda_2 \ge \ldots \ge \lambda_n$ be its eigenvalues. Also, let $M$ denote the Gram matrix corresponding to $\Ce$. Recall that by definition, $M$ and $A$ are related via \eqref{eq_gram_adjacency}.
Using $n \ge r+2$ (which we can since otherwise $n \le r+1$ is smaller than our second term), we may apply \Cref{lem_beta_upper} to conclude that $\beta = \beta(G) = \frac{1-\alpha}{2\alpha}$ is an eigenvalue of $A$ and hence $\lambda_1 \geq \beta$.

If $\lambda_1 = \beta$, then $(1-\alpha)I_n - 2\alpha A$ and $\alpha J_n$ are both positive semidefinite, so \eqref{eq_gram_adjacency} implies that the intersection of their nullspaces is the nullspace of $M$. Moreover, the Perron--Frobenius theorem implies that there exists a nonzero vector in the nullspace of $(1-\alpha)I_n - 2\alpha A$ with all nonnegative coordinates. This guarantees that it is not in the nullspace of $\alpha J_n$ and thus the dimension of the nullspace of $M$ is strictly smaller than the dimension of the nullspace of $(1-\alpha)I_n - 2\alpha A$. 
The eigenspace of $\beta$ as an eigenvalue of $A$ equals the nullspace of $(1-\alpha)I_n - 2\alpha A$. Putting all of this together and using the rank--nullity theorem combined with the fact that rank of $M$ is at most $r$ (so its nullspace has dimension at least $n-r$), we get 
\begin{equation}\label{eq:multiplicity-bound-strong}
    n -r \le m_G(\beta) - 1.
\end{equation}
Let $G_1,\ldots,G_t$ be the connected components of $G$. By our case assumption, $\beta \ge \lambda_1(G_i)$ for all $i$. Let $s\ge 1$ be the number of components $G_i$ for which the equality holds.
Observe that any such $G_i$ must have at least $\beta+1$ vertices (otherwise, its maximum degree and hence also its largest eigenvalue is less than $\beta$) and so $n \ge s(\beta+1)$.
Since each $G_i$ is connected, the Perron--Frobenius Theorem guarantees that its largest eigenvalue has multiplicity one and thus, $s=m_G(\beta)$. Putting these observations together with \eqref{eq:multiplicity-bound-strong}, we get 
$$n\ge m_G(\beta)(\beta+1)\ge (n-r+1)(\beta+1) \implies n \le (r-1) \cdot (1+1/\beta),$$ 
which gives the first term in the desired bound. 

We now turn to the case where $\lambda_1>\beta$. 
Since $\lambda_2 \leq \beta$ and $\beta$ is an eigenvalue of $G$, we must have $\lambda_2 = \beta$. 
Note that the eigenspace of $\beta = \frac{1-\alpha}{2\alpha}$ equals the nullspace of $(1-\alpha)I_n - 2\alpha A=M-\alpha J_n$. Since the rank of $M$ is at most $r$ and the rank of $\alpha J_n$ equals one, we conclude that the rank of $M-\alpha J_n$ is at most $r+1$, so again by the rank--nullity theorem, the dimension of its nullspace is at least $n-r-1$. Therefore, in this case, we obtain a slightly weaker variant of \eqref{eq:multiplicity-bound-strong}:
\begin{equation}\label{eq:multiplicity-bound}
    n - r\le m_G(\beta) + 1.
\end{equation}

Let us again denote by $G_1,\ldots, G_t$, the connected components of $G$ ordered so that $\lambda_1(G_1)=\lambda_1$. 

We now claim that $\lambda_1(G_i) < \beta$ for all $i \geq 2$. The argument is a variant of \Cref{obs:disjoint-supports} that applies to the disconnected graph $G$. Let $i \geq 2$ be fixed and suppose, for the sake of contradiction, that $\lambda_1(G_i) \ge \beta$. As in the proof of \Cref{obs:disjoint-supports}, let $\textbf{v}, \textbf{u} \in \R^{V(G)}$ be unit length eigenvectors for $\lambda_1(G_1), \lambda_1(G_i)$, respectively, and observe that $\langle\textbf{v},\textbf{u}\rangle=0$. Let $a,b \in \R$ be chosen so that the vector $\textbf{x}=a\textbf{v}+b\textbf{u}$ is orthogonal to $\one$ and $a^2+b^2=1$. Via the Perron-Frobenius theorem, we may assume that all entries of $\textbf{v}$ and $\textbf{u}$ are nonnegative, so that neither $\textbf{v}$ nor $\textbf{u}$ are orthogonal to $\one$, and hence $a,b \neq 0$. Note that $\textbf{v}^{\intercal}A\textbf{u}=0$ since $G_1$ and $G_i$ are disconnected, so that using $\lambda_1(G_1) > \beta$, we conclude
$$\beta \geq \textbf{x}^{\intercal}A\textbf{x}=a^2 \textbf{v}^{\intercal}A\textbf{v} +b^2\textbf{u}^{\intercal}A\textbf{u}+2ab\textbf{v}^{\intercal}A\textbf{u}=a^2\lambda_1(G_1)+b^2\lambda_1(G_i)>\beta(a^2+b^2)=\beta,$$
a contradiction.

Having established that $\lambda_1(G_i) < \beta = \lambda_2$ for all $i \geq 2$, we conclude that $\lambda_2(G_1) = \lambda_2$ and moreover, $m_G(\beta) = m_{G_1}\left(\lambda_2\right) \le \MC\left(n, \frac{1-\alpha}{2\alpha}, \frac{6}{\alpha^4} \right)$, so that \eqref{eq:multiplicity-bound} gives the second term in the desired inequality.

For the moreover part of the lemma, note that in the above argument, the case $\lambda_1 = \beta$ cannot occur if there is no graph with largest eigenvalue equal to $\beta.$ 
\end{proof}

We are now ready to present a proof of a precise version of \Cref{thm:main} with an explicit $O(1)$ term of $20$. We note that we made no effort to optimize this constant and rather focused on making the argument as general and easy to read as possible. By being significantly more careful and tweaking multiple parts of the argument, we can improve this constant. In the instance relevant for \Cref{cor:main}, we can even improve it to $4+o(1)$ and believe that $2+o(1)$ is a natural barrier for our current methods. 

\begin{thm}\label{thm:main-precise}
Given $0< \alpha <1 $ and an integer $r\ge 2^{1/\alpha^{20}},$ the number of $\alpha$-equiangular lines in $\mathbb{R}^r$ is at most
$$
    \floor{\left(1+\frac{2\alpha}{1-\alpha}\right)(r-1)}. 
    $$
\end{thm}
\begin{proof}
Let $n$ denote the maximum number of $\alpha$-equiangular lines in $\mathbb{R}^r$. We may assume $n \ge r+2$, since otherwise $n \le r+1 \le r-1+\frac{2\alpha}{1-\alpha}(r-1)$ and the desired inequality holds.

Suppose for now that $\alpha \le 1/3$, so that $\lambda_2 = \frac{1-\alpha}{2\alpha} \ge 1.$ Using \Cref{thm:multiplicity-approx-bound} with $r=4$ and $B=\lambda_2+3$, we get $$\MC(n,\lambda_2,6/\alpha^4) \le \frac{n}{B-1} \le \frac{n}{\lambda_2+2},$$ where the requirements of the theorem are satisfied since in a connected graph, any vertex of degree larger than $\lambda_2-1$ will have at least $\lambda_2+3$ vertices within distance four from it and in addition, $2^9 r \log \Delta \log B \lambda_2^{4r+2} \le  2^{11} \left( 4\log \lambda_2 + 9 \right)\log (\lambda_2+3) \lambda_2^{18}\le 2^{16}\lambda_2^{20} < (1/\alpha)^{20} \le \log n.$ Now observe that \Cref{lem:lines-to-multiplicity} implies that either $n
\le \left(1+\frac{2\alpha}{1-\alpha}\right)(r-1)$ or $$n \le r + 1 + \MC(n,\lambda_2,6/\alpha^4) \le r-1+\frac{n}{\lambda_2+1}  \implies n \le \left(1 + \frac{1}{\lambda_2} \right)(r-1) = \left(1+\frac{2\alpha}{1-\alpha}\right)(r-1).$$ In either case, we obtain $n \le \left(1+\frac{2\alpha}{1-\alpha}\right)(r-1)$ and thus, the desired bound follows from the integrality of $n$.

Otherwise if $\alpha>1/3$, then we have $\lambda_2<1$ and thus \Cref{obs:disjoint-supports} implies that a connected graph $G$ with second eigenvalue $\lambda_2$ can not contain an induced matching of size $2$. Therefore, if we fix an edge $vu$, then outside of $\{v,u\} \cup N(u) \cup N(v)$ there can be no edges. This implies that the number of vertices of $G$ is at most $2\Delta(G)^2$. Hence, 
$$\MC(n,\lambda_2,6/\alpha^4) \le 2(6/\alpha^4)^2 \le \frac{2\alpha}{1-\alpha} \cdot (r-1)-2$$
and so, the desired bound follows via \Cref{lem:lines-to-multiplicity}.
\end{proof}

As mentioned in the introduction, the bound in this theorem is tight whenever $\alpha=\frac{1}{2k-1}$ (so $\beta=\frac{1-\alpha}{2\alpha}=k-1$) for any $k \in \mathbb{N}$. Indeed, there is a well-known construction which achieves equality in the case $\lambda_1=\beta$ of the above argument. This construction was pointed out by Bukh \cite{B17}, although it already follows from a result of Greaves, Koolen, Munemasa, and Sz\"oll\H{o}si \cite{GKMS16} (and is implicit in \cite{LS73} when they consider the case of $k=2$). We start with letting $G$ consist of the disjoint union of $t = \left \lfloor \frac{r-1}{k-1} \right \rfloor$ copies of a clique on $k$ vertices and $r - t(k-1) - 1$ isolated vertices. $A = A_G$ is then an $n \times n$ matrix where $$n = tk + r - t(k-1) - 1 = r - 1 + t = \left \lfloor \left(1 + \frac{1}{\beta} \right) (r-1) \right \rfloor,$$ it has $t$ eigenvalues being $k-1$, and all remaining eigenvalues are either $0$ or $-1$. Therefore, if we define $M = (1-\alpha)I + \alpha J - 2 \alpha A$, then for any nonzero vector $\textbf{x}$, we have $$\textbf{x}^\intercal M \textbf{x} = (1-\alpha)||\textbf{x}||^2 + \alpha \inprod{\textbf{x}}{\one}^2 - 2\alpha \textbf{x}^\intercal A \textbf{x} \geq \left((1-\alpha) - 2\alpha(k-1)\right) ||\textbf{x}||^2 = 0$$ with equality if and only if $\textbf{x}$ is orthogonal to $\one$ and it lies in the eigenspace of $A$ corresponding to $k-1$. Thus, $M$ is positive semidefinite, and its nullspace has dimension at least $t-1$, so its rank is at most $n - t + 1 = r$. It follows that $M$ must be the Gram matrix of some vectors $v_1, \ldots, v_n \in \R^r$ and moreover, the diagonal entries of $M$ are all $1$ and all off-diagonal entries are in $\{\pm \alpha\}$, so $\{v_1, \ldots, v_n\}$ is the desired spherical $\{\pm \alpha\}$-code. 

We conclude the section with a proof of a more precise version of \Cref{thm:superpolynomial}.

\begin{thm}\label{thm:superpolynomial-precise}
    Given $0< \alpha <1$ and an integer $r\ge 1/\alpha^{4},$ the number of $\alpha$-equiangular lines in $\mathbb{R}^r$ is at most
$$
r+r \cdot \max\left\{\frac{66\log^2 \log_{1/\alpha} r}{\log_{1/\alpha} r}, \: 17\alpha \log \frac{2}{\alpha} \right\}.
$$ 
\end{thm}
\begin{proof}
Let $n$ denote the maximum number of $\alpha$-equiangular lines in $\mathbb{R}^r$. From \cite[Theorem 1.2]{igor-paper} we have an upper bound of $(2+3\alpha)r$ which is smaller than the second term in our maximum unless $\alpha<2^{-7}$. Moreover, if $\alpha<2^{-7}$ then $(2+3\alpha)r$ is also smaller than the first term unless $(1+3\alpha)\log_{1/\alpha} r > 66 \log^2 \log_{1/\alpha} r \implies \log_{1/\alpha} r > 64 \log^2 \log_{1/\alpha} r$.

We may assume $n \ge r+2$, since otherwise $n \le r+1 \le r+2\alpha r \log \frac{2}{\alpha}$ and the desired inequality holds. Next, we apply \Cref{lem:lines-to-multiplicity}. Note that $\frac{2\alpha}{1-\alpha} \le 2\alpha \log \frac{2}{\alpha}$, so if the first term in \Cref{lem:lines-to-multiplicity} is larger, then we are done. Hence, we may assume that the second term is larger, so that we have
\[
n \leq r+1 + \MC(n,\lambda_2,6/\alpha^4),
\]
where $\MC(n,\lambda_2,6/\alpha^4)$ is the second eigenvalue multiplicity of an $n$-vertex graph $G$ with second eigenvalue $\lambda_2=\frac{1-\alpha}{2\alpha}$ and maximum degree $\Delta\le 6/\alpha^4$. To bound this multiplicity, we may apply \Cref{thm:multiplicity-combined}, which implies either that
$$n \le r+1 + \frac{4n\log (\lambda_2+1)}{\lambda_2} = r+1 + \frac{8n\alpha \log \frac{1+\alpha}{2\alpha}}{1-\alpha}\le r+8n \alpha \log \frac{2}{\alpha} \implies n \le r + r \cdot 16 \alpha \log \frac{2}{\alpha},$$
where the implication uses $8 \alpha \log \frac{2}{\alpha} \le \frac12$, or otherwise, that
$$n \le r+1 + \frac{4n \log^2 (1+\log_{\Delta} n)}{\log_{\Delta} n}
\le  r + \frac{32 n \log^2 \log_{1/\alpha} r}{\log_{1/\alpha} r}
\implies n \le r+r \cdot \frac{64 \log^2 \log_{1/\alpha} r}{\log_{1/\alpha} r},$$
where the second inequality uses $\log_{\Delta} n \ge \frac{\log n}{\log (6/\alpha^4)}\ge \frac{\log r}{8\log (1/\alpha)}=\frac18\log_{1/\alpha} r$ and the implication uses the fact that $\frac{32\log^2 \log_{1/\alpha}{r}}{\log_{1/\alpha} r} \leq \frac{1}{2}$.
\end{proof}

\section{Concluding remarks}
In this paper, we determined that the maximum number of equiangular lines in $\mathbb{R}^r$ with common angle $\arccos \alpha$ is $r + o(r)$ provided that $r$ is at least superpolynomial in $1/\alpha \to \infty.$ It would be interesting to determine whether having $r$ be a large polynomial in $1/\alpha$ already suffices for this result.

We also precisely determined the maximum number of equiangular lines in $\mathbb{R}^r$ with common angle $\frac{1}{2k-1}$ for $k \in \mathbb{N}$ where $r$ is at least exponential in $k^{O(1)}$, improving on the previous, doubly exponential requirements from \cite{igor-paper,yufei-paper}. The most immediate open question is whether the same behavior extends to the subexponential regime as well. Here, \Cref{thm:superpolynomial-precise} gives an upper bound of the form $r+\frac{r}{(\log_{k} r)^{1-o(1)}}$, which improves the best-known bounds already when $r$ is a large polynomial in $k$, but is not strong enough to match the best-known constructions.

\Cref{cor:main} demonstrates that \Cref{thm:main} is tight when 
$1/\alpha$ is an odd integer. We note that one can also use our tools to improve \Cref{thm:main} when this is not the case. 
In fact, we can determine the correct answer in the exponential regime for an infinite family of other $\alpha$, namely whenever the spectral radius order of $\frac{1-\alpha}{2\alpha}$ is small. This only requires one to take a slightly larger $r$ in the application of \Cref{thm:multiplicity-approx-bound} at the end of the proof of \Cref{thm:main-precise}. The cost of this is an increase in the required lower bound on the dimension. However, so long as the spectral radius order remains close to $\frac{1}{2\alpha}$, the requirement remains exponential.

Our results are based on significant improvements to the best-known bounds on the second eigenvalue multiplicity of a graph. For sufficiently dense graphs, our bound improves upon a result of McKenzie, Rasmussen, and Srivastava \cite{tcs-paper}, extending it from regular to arbitrary graphs and substantially extending the range of density for which the bound is non-trivial. For sparser graphs, we prove the following appealing inequality for an $n$ vertex graph $G$ with second eigenvalue $\lambda_2$.
\begin{equation}\label{eq:dream}
    m_G(\lambda_2)\le \frac{n}{\lambda_2+1}+n^{o(1)},
\end{equation} 
provided that the maximum degree of $G$ is upper bounded by a polylogarithmic function in $n$. It is, however, possible that this bound holds already when the maximum degree is only required to be at most $n^{o(1)}$. If this were the case, it would prove an only slightly weaker version of \Cref{thm:main} already when $r$ is polynomial in $1/\alpha$. 
Moreover, a slightly stronger bound in \eqref{eq:dream} for sufficiently dense connected graphs would lead to a full solution of the equiangular lines problem %raised by Lemmens and Seidel in 1973 
when $1/\alpha$ is an odd integer already from the polynomial dimension onwards.
Indeed, one can not in general improve upon \eqref{eq:dream} as the graph might be a disjoint union of cliques on $\lambda_2+1$ vertices, whereas our \Cref{cor:multiplicity-bound} shows that one can improve upon it when the graph is assumed to be connected and sufficiently sparse. 

This inequality is reminiscent of a classical conjecture in spectral graph theory, due to Powers \cite{powers89} from 1989, which asserts that if $\lambda_1 \ge \ldots \ge \lambda_n$ are the eigenvalues of an $n$ vertex graph, then $\lambda_i \le \floor{n/i}$. This appealing conjecture has unfortunately been disproved for $i \ge 4$. But a number of variants have been raised over the years with, in particular, Nikiforov \cite{nikiforov} recently proving a number of results surrounding the problem. See the survey \cite{spectral-survey} for more details.

McKenzie, Rasmussen, and Srivastava \cite{tcs-paper} proved that the second eigenvalue multiplicity of the \emph{normalized} adjacency matrix of a connected graph with maximum degree $\Delta$ is at most $\frac{n \Delta^{7/5}}{(\log n)^{1/5-o(1)}}$. While there is generally no direct relation between normalized and usual adjacency matrices, if the underlying graph is regular, they are the same up to a scalar multiple. This implies that this bound also holds for regular graphs and the usual adjacency matrix. Note that when $\lambda_2$ is small (so in the superexponential regime), this bound improves upon \eqref{eq:dream}. It is an intriguing question to determine if one can combine their approach with ours in this regime to remove the regularity assumption. However, unfortunately, similar types of obstructions cause issues for both of our arguments.

As mentioned in \Cref{sec:second-eigenvalue}, Letrouit and Machado \cite{LM24} have recently adapted the approach of Jiang, Tidor, Yao, Zhang, and Zhao \cite{yufei-paper} to the setting of negatively curved two-dimensional Riemannian manifolds. This allowed them to establish a bound on the second eigenvalue multiplicity of a corresponding Laplace-Beltrami operator, thereby making progress on a conjecture of Colin de Verdi\`ere \cite{C87}. Since our approach substantially improves the multiplicity bound of \cite{yufei-paper} in the setting of graphs, it would be interesting to see if it can also be generalized to the setting of Riemannian manifolds.

\textbf{Acknowledgements.}
We want to thank Alp M\"uyesser, Hung-Hsun Hans Yu, Theo McKenzie, Mehtaab Sawhney, Varun Sivashankar, Carl Schildkraut, Benny Sudakov, and Shengtong Zhang for useful comments and discussions.
The first author would like to thank the Mathematics Department at Princeton University for hosting him while part of this work took place. The second author would like to gratefully acknowledge the support of the Oswald Veblen Fund and the Institute for Advanced Study in Princeton. 

%\textbf{Remark.} In an earlier version of the paper we claimed in the moreover part of \Cref{thm:multiplicity-main} a bound without the error term which is sadly incorrect. The issue with our original proof was in the density assumption we make not being hereditary so one can plant a very dense small graph where such an inequality is simply not true and pad it with cliques to obtain a very sparse graph at a global level. The example we now include in the introduction shows that one can indeed disprove it by a tiny additive error term. %We note though that the error term we prove is tiny in comparison to the main term and in particular can be avoided in graphs coming from the equiangular setting (even without the connectivity assumption).

%\textbf{Data availability.} This project involved no data.

%\textbf{Conflict of interest.} The authors have no conflicts of interest related to the work in this paper.

\providecommand{\MR}[1]{}
\providecommand{\MRhref}[2]{%
  \href{http://www.ams.org/mathscinet-getitem?mr=#1}{#2}
}

%   \bibliographystyle{amsplain_initials_nobysame}
%   \bibliography{ref}

\begin{thebibliography}{10}

\bibitem{ABGHM22}
M.~Appleby, I.~Bengtsson, M.~Grassl, M.~Harrison, and G.~McConnell, \emph{{SIC-POVMs} from stark units: Prime dimensions $n^2 + 3$}, J. Math. Phys. \textbf{63} (2022), no.~11, 112205.

\bibitem{AFMY17}
M.~Appleby, S.~Flammia, G.~McConnell, and J.~Yard, \emph{{SICs} and algebraic number theory}, Found. Phys. \textbf{47} (2017), no.~8, 1042--1059.

\bibitem{AFMY20}
M.~Appleby, S.~Flammia, G.~McConnell, and J.~Yard, \emph{Generating ray class fields of real quadratic fields via complex equiangular lines}, Acta Arith. \textbf{192} (2020), no.~3, 211--233. \MR{4048602}

\bibitem{igor-paper}
I.~Balla, \emph{Equiangular lines via matrix projection}, preprint arXiv:2110.15842 (2021).

\bibitem{igor-felix-paper}
I.~Balla, F.~Dr\"axler, P.~Keevash, and B.~Sudakov, \emph{Equiangular lines and spherical codes in {E}uclidean space}, Invent. Math. \textbf{211} (2018), no.~1, 179--212. \MR{3742757}

\bibitem{B19}
G.~Basso, \emph{Computation of maximal projection constants}, J. Funct. Anal. \textbf{277} (2019), no.~10, 3560--3585. \MR{4001080}

\bibitem{B17}
I.~Bengtsson, \emph{The number behind the simplest {SIC-POVM}}, Found. Phys. \textbf{47} (2017), no.~8, 1031--1041.

\bibitem{B53}
L.~M. Blumenthal, \emph{Theory and applications of distance geometry}, Clarendon Press, Oxford, 1953.

\bibitem{B16}
B.~Bukh, \emph{Bounds on equiangular lines and on related spherical codes}, SIAM J. Discrete Math. \textbf{30} (2016), no.~1, 549--554. \MR{3477753}

\bibitem{CKLY22}
M.-Y. Cao, J.~H. Koolen, Y.-C.~R. Lin, and W.-H. Yu, \emph{The {L}emmens-{S}eidel conjecture and forbidden subgraphs}, J. Combin. Theory Ser. A \textbf{185} (2022), Paper No. 105538, 28. \MR{4316717}

\bibitem{CKP13}
P.~G. Casazza, G.~Kutyniok, and F.~Philipp, \emph{Introduction to finite frame theory}, Finite frames, Appl. Numer. Harmon. Anal., Birkh\"auser/Springer, New York, 2013, pp.~1--53. \MR{2964006}

\bibitem{CFS02}
C.~M. Caves, C.~A. Fuchs, and R.~Schack, \emph{Unknown quantum states: the quantum de {F}inetti representation}, J. Math. Phys. \textbf{43} (2002), no.~9, 4537--4559, Quantum information theory. \MR{1924454}

\bibitem{colding}
T.~H. Colding and W.~P. Minicozzi, II, \emph{Harmonic functions on manifolds}, Ann. of Math. (2) \textbf{146} (1997), no.~3, 725--747. \MR{1491451}

\bibitem{C87}
Y.~Colin~de Verdi\`ere, \emph{Construction de laplaciens dont une partie finie du spectre est donn\'ee}, Ann. Sci. \'Ecole Norm. Sup. (4) \textbf{20} (1987), no.~4, 599--615. \MR{932800}

\bibitem{C73}
H.~S.~M. Coxeter, \emph{Regular polytopes}, 3rd ed., Dover Publications, 1973.

\bibitem{DL23}
B.~Der\c{e}gowska and B.~Lewandowska, \emph{A simple proof of the {G}r{\"u}nbaum conjecture}, J. Funct. Anal. \textbf{285} (2023), no.~2, Paper No. 109950, 8. \MR{4575685}

\bibitem{FS17}
S.~Foucart and L.~Skrzypek, \emph{On maximal relative projection constants}, J. Math. Anal. Appl. \textbf{447} (2017), no.~1, 309--328. \MR{3566474}

\bibitem{FHS17}
C.~A. Fuchs, M.~C. Hoang, and B.~C. Stacey, \emph{The {SIC} question: History and state of play}, Axioms \textbf{6} (2017), no.~3, 21.

\bibitem{FS19}
C.~A. Fuchs and B.~C. Stacey, \emph{{QBism}: Quantum theory as a hero's handbook}, Foundations of Quantum Theory, IOS Press, 2019, pp.~133--202.

\bibitem{FS03}
C.~A. Fuchs and M.~Sasaki, \emph{Squeezing quantum information through a classical channel: measuring the ``quantumness'' of a set of quantum states}, Quantum Inf. Comput. \textbf{3} (2003), no.~5, 377--404. \MR{2006047}

\bibitem{GY18}
A.~Glazyrin and W.-H. Yu, \emph{Upper bounds for {$s$}-distance sets and equiangular lines}, Adv. Math. \textbf{330} (2018), 810--833. \MR{3787558}

\bibitem{GR01}
C.~Godsil and G.~F. Royle, \emph{Algebraic graph theory}, vol. 207, Springer Science \& Business Media, 2001.

\bibitem{GKMS16}
G.~Greaves, J.~H. Koolen, A.~Munemasa, and F.~Sz\"oll\H~osi, \emph{Equiangular lines in {E}uclidean spaces}, J. Combin. Theory Ser. A \textbf{138} (2016), 208--235. \MR{3423477}

\bibitem{gromov}
M.~Gromov, \emph{Groups of polynomial growth and expanding maps}, Inst. Hautes \'Etudes Sci. Publ. Math. (1981), no.~53, 53--73. \MR{623534}

\bibitem{HS47}
J.~Haantjes and J.~J. Seidel, \emph{The congruence order of the elliptic plane}, Indagationes Mathematicae \textbf{9} (1947), 403--405.

\bibitem{HC17}
J.~I. Haas and P.~G. Casazza, \emph{On the structures of {Grassmannian} frames}, 2017 International Conference on Sampling Theory and Applications (SampTA), {IEEE}, July 2017, pp.~377--380.

\bibitem{schildkraut-paper}
M.~Haiman, C.~Schildkraut, S.~Zhang, and Y.~Zhao, \emph{Graphs with high second eigenvalue multiplicity}, Bull. Lond. Math. Soc. \textbf{54} (2022), no.~5, 1630--1652. \MR{4499595}

\bibitem{JP20}
Z.~Jiang and A.~Polyanskii, \emph{Forbidden subgraphs for graphs of bounded spectral radius, with applications to equiangular lines}, Israel J. Math. \textbf{236} (2020), no.~1, 393--421. \MR{4093893}

\bibitem{yufei-paper}
Z.~Jiang, J.~Tidor, Y.~Yao, S.~Zhang, and Y.~Zhao, \emph{Equiangular lines with a fixed angle}, Ann. of Math. (2) \textbf{194} (2021), no.~3, 729--743. \MR{4334975}

\bibitem{kleiner}
B.~Kleiner, \emph{A new proof of {G}romov's theorem on groups of polynomial growth}, J. Amer. Math. Soc. \textbf{23} (2010), no.~3, 815--829. \MR{2629989}

\bibitem{KLL83}
H.~K\"onig, D.~Lewis, and P.~K. Lin, \emph{Finite dimensional projection constants}, Studia Math. \textbf{75} (1983), no.~3, 341--358.

\bibitem{lee2008eigenvalue}
J.~R. Lee and Y.~Makarychev, \emph{Eigenvalue multiplicity and volume growth}, preprint arXiv:0806.1745 (2008).

\bibitem{LS73}
P.~W.~H. Lemmens and J.~J. Seidel, \emph{Equiangular lines}, J. Algebra \textbf{24} (1973), 494--512.

\bibitem{LM24}
C.~Letrouit and S.~Machado, \emph{Maximal multiplicity of {L}aplacian eigenvalues in negatively curved surfaces}, Geom. Funct. Anal. \textbf{34} (2024), no.~5, 1609--1645. \MR{4792842}

\bibitem{spectral-survey}
L.~Liu and B.~Ning, \emph{Unsolved problems in spectral graph theory}, Oper. Res. Trans. \textbf{27} (2023), no.~4, 33--60. \MR{4742366}

\bibitem{tcs-paper}
T.~McKenzie, P.~M.~R. Rasmussen, and N.~Srivastava, \emph{Support of closed walks and second eigenvalue multiplicity of graphs}, S{TOC} '21---{P}roceedings of the 53rd {A}nnual {ACM} {SIGACT} {S}ymposium on {T}heory of {C}omputing, ACM, New York, 2021, pp.~396--407. \MR{4398851}

\bibitem{N89}
A.~Neumaier, \emph{Graph representations, two-distance sets, and equiangular lines}, Linear Algebra Appl. \textbf{114/115} (1989), 141--156. \MR{986870}

\bibitem{nikiforov-eigenvalue-growth}
V.~Nikiforov, \emph{Revisiting two classical results on graph spectra}, Electron. J. Combin. \textbf{14} (2007), no.~1, Research Paper 14, 7. \MR{2285818}

\bibitem{nikiforov}
V.~Nikiforov, \emph{Extrema of graph eigenvalues}, Linear Algebra Appl. \textbf{482} (2015), 158--190. \MR{3365272}

\bibitem{powers89}
D.~L. Powers, \emph{Bounds on graph eigenvalues}, Linear Algebra Appl. \textbf{117} (1989), 1--6. \MR{993025}

\bibitem{RBSC04}
J.~M. Renes, R.~Blume-Kohout, A.~J. Scott, and C.~M. Caves, \emph{Symmetric informationally complete quantum measurements}, Jour. Math. Phys. \textbf{45} (2004), no.~6, 2171--2180.

\bibitem{HS03}
T.~Strohmer and R.~W. Heath, Jr., \emph{Grassmannian frames with applications to coding and communication}, Appl. Comput. Harmon. Anal. \textbf{14} (2003), no.~3, 257--275. \MR{1984549}

\bibitem{Y17}
W.~H. Yu, \emph{New bounds for equiangular lines and spherical two-distance sets}, SIAM J. Discrete Math. \textbf{31} (2017), no.~2, 908--917.

\bibitem{Z11}
G.~Zauner, \emph{Grundz\"{u}ge einer nichtkommutativen {D}esigntheorie}, Ph.D. thesis, University of Vienna, 1999, Published in English translation: Int. J. Quantum Inf. 9, 2011.

\bibitem{notices}
Y.~Zhao, \emph{Equiangular lines and eigenvalue multiplicities}, Notices Amer. Math. Soc \textbf{71} (2024), no.~9, 1151.

\end{thebibliography}

\providecommand{\bysame}{\leavevmode\hbox to3em{\hrulefill}\thinspace}
\providecommand{\MR}{\relax\ifhmode\unskip\space\fi MR }
\providecommand{\MRhref}[2]{%
  \href{http://www.ams.org/mathscinet-getitem?mr=#1}{#2}
}
\providecommand{\href}[2]{#2}

\end{document}